\documentclass[reqno,12pt,letterpaper]{amsart}
\usepackage{amsmath,amssymb,amsthm,graphicx,mathrsfs,url}
\usepackage[usenames,dvipsnames]{color}
\usepackage{hyperref}
\usepackage[super]{nth}
\usepackage[open, openlevel=2, depth=3, atend]{bookmark}
\hypersetup{pdfstartview=XYZ}

\def\?[#1]{\textbf{[#1]}\marginpar{\Large{\textbf{??}}}}

\newtheorem{thm}{Theorem}
\newtheorem{prop}{Proposition}[section]
\newtheorem{Def}[prop]{Definition}

\newtheorem{lem}[prop]{Lemma}
\newtheorem{cor}[prop]{Corollary}
\newtheorem{rem}{Remark}
\numberwithin{equation}{section}

\newcommand{\mc}{\mathcal}
\newcommand{\rr}{\mathbb{R}}

\newcommand{\cc}{\mathbb{C}}
\newcommand{\hh}{\mathbb{H}}

\newcommand{\la}{\lambda}
\newcommand{\eps}{\epsilon}

\newcommand{\pl}{\partial}
\newcommand{\x}{\times}

\newcommand{\til}{\widetilde}
\newcommand{\bbar}{\overline}

\newcommand{\cjd}{\rangle}
\newcommand{\cjg}{\langle}

\newcommand{\demi}{\tfrac{1}{2}}

\DeclareMathOperator{\Res}{Res}

\DeclareMathOperator{\Eig}{Eig}

\let\Im=\Imag

\let\Re=\Real

\DeclareMathOperator{\WF}{WF}
\DeclareMathOperator{\Tr}{Tr}

\usepackage{array}
\usepackage{graphicx}
\usepackage[english]{babel}
\usepackage[utf8]{inputenc}

\usepackage[arrow, matrix, curve]{xy}

\usepackage{tikz}

\usepackage{tikz-cd}
\newcommand{\tu}[1]{\textup{#1}}

\newcommand{\Abb}[4]{\left\{ \begin{array}{ccc}
                               #1 & \rightarrow &#2\\
			       #3 &\mapsto &#4
                               \end{array}\right.}

\renewcommand{\Re}{\operatorname{Re}}
\renewcommand{\Im}{\operatorname{Im}}

\newcommand{\N}{\mathbb{N}}
\newcommand{\R}{\mathbb R}
\newcommand{\Z}{\mathbb Z}
\newcommand{\C}{\mathbb{C}}

\newcommand{\intd}{\mathop{}\!\mathrm{d}}
\newcommand{\g}{\mathfrak g}

\newcommand{\p}{\mathfrak p}

\newcommand{\n}{\mathfrak n}
\renewcommand{\a}{\mathfrak a}
\newcommand{\Ad}{\operatorname{Ad}}
\newcommand{\cpic}[5]{\begin{figure}[#5]
 \centering \makebox[0cm]{\includegraphics[width= #1]{./#2}}
\caption{#3}
\label{fig:#4}
\end{figure} }

\newcommand{\bn}{{\overline{n}}}
\newcommand{\bN}{{\overline{N}}}
\newcommand{\hspc}{\mathbf{H}}
\newcommand{\bhspc}{\partial \mathbf{H}}
\newcommand{\cH}{{\check{H}_0}}

\title[Invariant Ruelle densities]{High frequency limits for 
invariant Ruelle densities}

 \author{Colin Guillarmou}
 \email{colin.guillarmou@math.u-psud.fr}
\address{CNRS, Universit\'e Paris-Sud, D\'epartement de Math\'ematiques, 91400
Orsay, France}
 \author{Joachim Hilgert}
\email{hilgert@math.upb.de}
 \address{Universit\"at Paderborn, Warburgerstr. 100, 33098 Paderborn, Germany}
 \author{Tobias Weich}
\email{weich@math.upb.de}
\address{Universit\"at Paderborn, Warburgerstr. 100, 33098 Paderborn, Germany}

\date{\today}

\begin{document}

\begin{abstract}
We establish an equidistribution result for Ruelle resonant states on compact locally symmetric spaces of rank one. More precisely, we prove that among the first band Ruelle resonances there is a density one subsequence such that the respective products of resonant and co-resonant states converge weakly to the Liouville measure. We prove this result by establishing an explicit quantum-classical correspondence between eigenspaces of the scalar Laplacian and the resonant states of the first band of Ruelle resonances which also leads to a new description of Patterson-Sullivan distributions. 
\end{abstract}

\maketitle
%\tableofcontents

\section{Introduction}

Let $X$ be a smooth Anosov vector field on a compact Riemannian manifold $\mathcal M$. Then the resolvent $R(\lambda):=(-X-\lambda)^{-1}:L^2(\mathcal M)\to L^2(\mathcal M)$ is
holomorphic for $\tu{Re}(\lambda)\gg 0$ and has a meromorphic continuation to $\C$ 
\cite{Liv04, BL07,FSj11,DZ16} as a family of continuous operators $R(\lambda):C^\infty(\mathcal M) \to \mathcal D'(\mathcal M)$. The poles of this meromorphic continuation are called \emph{Ruelle resonances}. 
Given a pole $\lambda_0$ of the resolvent, (minus) the residue of the resolvent $R(\la)$ at $\la_0$
 is a finite rank operator $\Pi_{\lambda_0}$ and we call its range $\tu{Ran}(\Pi_{\lambda_0})\subset \mathcal D'(\mathcal M)$ the space of \emph{generalized Ruelle resonant states}. The generalized resonant states are known to be distributions which are supported on the whole manifold $\mathcal M$ \cite{Wei17} and they lie in $\ker (-X-\la_0)^J$ for some $J\geq 1$, and $J=1$ if and only if there are no Jordan blocks. 
Furthermore, to each Ruelle resonance $\lambda_0$ we can define a canonical generalized density
in the following way: Because the wavefront set of $\Pi_{\lambda_0}$ is precisely known \cite[Theorem 1.7]{FSj11}\cite[Proposition 3.3]{DZ16}, there is a well defined notion of a trace of $\Pi_ {\lambda_0}$, the so called flat trace $\Tr^\flat$ and we can define the following continuous linear functional
\[
 \mathcal T_{\lambda_0}:\Abb{C^\infty(\mathcal M)}{\C}{f}{\Tr^\flat(f\Pi_{\lambda_0})}
\]
If $\lambda_0=0$, then the functional $\mathcal T_{\lambda_0}$ is given by the SRB measure\footnote{More precisely by an SRB measure, because the eigenvalue $\lambda_0$ might be degenerate and the SRB measure not unique.}, see \cite{BL07}. For a general resonance $\la_0\in \C$ these functionals are only distributional densities and not measures. 
As $X$ commutes with $\Pi_{\lambda_0}$, they are still invariant under the flow, i.e. $X\mc{T}_{\la_0}=0$, and we call them \emph{invariant Ruelle densities}. Note that these invariant Ruelle densities have also explicit expressions in terms of the Ruelle resonant states: In the simplest case of a first order pole of multiplicity one the invariant density is simply the distributional product of a resonant and co-resonant state, where \emph{co-resonant states} are resonant states for the flow in backward time. 

We want to study \emph{high frequency limits} (also called \emph{semiclassical limits}) of these invariant Ruelle distributions. For Ruelle resonances the reasonable notion of semiclassical limit is to fix a range in the real part $\tu{Re}(\lambda)>-C$ and consider $|\tu{Im}(\lambda)|\to\infty$. In this limit there have recently been established several results on the distributions of resonances such as Weyl laws \cite{FSj11, DDZ14, FT17b} or band structures 
\cite{FT13, FT17}. 

\textbf{High-frequency limits.} 
 If $G/K$ is a rank one Riemannian symmetric space of noncompact type, $\Gamma\subset G$ a co-compact torsion free discrete subgroup, then the locally symmetric space 
$\mathbf M:= \Gamma\backslash G/K$ is a compact Riemannian manifold of strictly negative curvature. Its geodesic flow on the unit sphere bundle $\mathcal M:= S\mathbf M$ is Anosov. Any compact manifold of constant negative curvature can be realized in this way, but the rank one locally symmetric spaces contain also families with nonconstant sectional curvature. In constant curvature the spectrum is known to obey an exact band structure \cite{DFG15, KW19a}, i.e. for any Ruelle resonance $\lambda_0$ one has either $\tu{Im}(\lambda_0)=0$, or $\tu{Re}(\lambda_0) \in -\rho - \N_0$, where $\rho>0$ is the positive constant that is associated to a Riemanian symmetric space by taking half the sum of its positive restricted roots. Our first result is the following equidistribution theorem.

\begin{thm}\label{thm: Main Theorem}
 Let $\mathbf M$ be a compact locally symmetric space of rank one, $\mathcal M=S\mathbf M$
 be the unit sphere bundle and $\intd\mu_{\tu{L}}$ the Liouville measure on $\mathcal M$.
 Let $r_n\in \R^+$ be such that $\lambda_n=-\rho+ir_n$ are the Ruelle resonances with $\Re(\lambda_n)=-\rho$, $\Im(\lambda_n)>0$ for the geodesic flow on $\mathcal M$. Then there exists a  subsequence $(r_{k_n})_{n>0}\subset (r_n)_{n>0}$, such that
 \begin{itemize}
  \item $\mathcal T_{-\rho + ir_{k_n}}$ converges weakly towards $\intd \mu_{\tu{L}}$ as $n\to\infty$.
 \item The subsequence is of density one, i.e.
 \[
   \lim_{N\to\infty}\frac{\sum_{k_n<N} \dim(\tu{Ran}\Pi_{-\rho + ir_{k_n}})}{\sum_{n<N} \dim(\tu{Ran}\Pi_{-\rho + ir_n})}=1.
 \]  
\end{itemize}
\end{thm}

To prove this result, we have to show an explicit  correspondence between the Ruelle resonant states on ${\rm Re}(\la)=-\rho$ and the eigenstates of the Laplacian $\Delta_{\bf M}$. This allows us to reduce the problem to a quantum ergodicity result for the Laplacian and use the Shnirelman-Zelditch-Colin de Verdi\`ere theorem \cite{Shn74, Zel87, CdV85}. In fact, we prove several results in this article which are of independent interest.

{\bf The quantum-classical correspondence (Theorem~\ref{thm:intertwining} and Theorem~\ref{thm:jordan}).} In \cite{DFG15, GHW18, Had18} it is shown that for geodesic flows on compact (resp. convex cocompact) constant negative curvature manifolds the Ruelle resonances are related to eigenvalues (resp. quantum resonances) for the Laplacian.
 A central ingredient in the proof is to establish an explicit bijection between the Ruelle resonant states in $\mathcal D'(S\mathbf M)$  that are killed by unstable derivatives and the eigenstates  of the Laplacian  $\Delta_{\bf M}$ on $\mathbf M$, at least for Ruelle resonances that are not in a certain exceptional set. The map from Ruelle resonant states to eigenfunctions of $\Delta_{\bf M}$ is given by the pushforward ${\pi_0}_*:\mc{D}'(S\mathbf M)\to \mc{D}'(\mathbf M)$, where $\pi_0:S{\bf M}\to {\bf M}$ is the projection onto the base. We extend this bijection to the setting of all compact locally symmetric  spaces of rank one in Theorem~\ref{thm:intertwining} and Theorem~\ref{thm:jordan}. To prove the bijection, one uses as in \cite{DFG15} a correspondence between  Ruelle (generalized) resonant states killed by the unstable derivates and the distributions on the boundary of $G/K$ that have a certain conformal equivariance by the group $\Gamma$, and the bijection between these distributions and $\Gamma$-invariant eigenfunctions of Laplacian on $G/K$, which thus descend to $\mathbf{M}$. The quantum-classical correspondence that we establish in this article form a crucial basis for the generalizations to vector bundles \cite{KW19a, KW20}. 
 
 We note that the correspondence between the distributions on the boundary and the Laplace eigenfunctions follows from works of \cite{He74, OS80, vBS87, Ot98, GO05} on the Poisson transformation. Some results related to the quantum-classical correspondence were obtained previously in \cite{FF03,Co05} for $G={\rm SL}_2(\rr)$ where the authors study distributions on $S\mathbf{M}$ invariant by the horocycle flow. In particular in \cite{Co05}, a relation is made between such distributions and conformally equivariant distributions on the boundary $\pl(G/K)$, similarly to what we discussed above.

{\bf A new description of Patterson-Sullivan  distributions (Theorem~\ref{thm:patterson_sullivan}).} In \cite{AZ07}  Anantharaman and Zelditch introduced Patterson-Sullivan distributions on compact hyperbolic surfaces. Given an eigenfunction of the Laplacian in $C^ \infty(\mathbf M)$, these distributions are distributions in  $\mathcal D'(S\mathbf M)$ which are invariant under the geodesic flow and become equivalent to usual semiclassical lifts such as Wigner distributions in the semiclassical limit. In \cite{HHS12} this construction has been generalized to compact higher rank locally symmetric spaces. Using the quantum-classical  correspondence we give in Theorem~\ref{thm:patterson_sullivan} a new description of these Patterson-Sullivan distributions for rank one spaces: 
given a Laplace eigenfunction $\varphi\in C^ \infty(\mathbf M)$ the quantum-classical correspondence allows us to associate a unique Ruelle resonant state $v\in \mathcal D'(S\mathbf M)$ as well as a Ruelle co-resonant state $v^* \in\mathcal D'(S\mathbf M)$. The Patterson-Sullivan distribution is then precisely given by the distributional product of $v\cdot v^*$ which is  well defined by a wavefront condition. 
 
{\bf A pairing formula (Theorem \ref{thm:pairing_formula}).} If the Ruelle resonance $\la_0$ associated to a 
 Patterson-Sullivan distribution is simple, then it is easy to check that the Patterson-Sullivan  distribution coincides with the invariant Ruelle density. If ${\rm Rank}(\Pi_{\la_0})>1$,
 one additionally needs a pairing formula  in order
 to express the invariant Ruelle density in terms of the Patterson-Sullivan distributions. 
The pairing formula relates pairings of resonant states with co-resonant states to pairings of the associated eigenfunctions of $\Delta_{\bf M}$.
Such pairing formulas have previously been proved in \cite{AZ07} for hyperbolic surfaces and in \cite{DFG15} for compact constant negative curvature manifolds using different methods. We extend them to all rank one cases in Theorem \ref{thm:pairing_formula}.
We follow the strategy of \cite{DFG15} but we emphasize that new difficulties appear due to the anisotropy of the Lyapunov exponents given by the fact that the curvature is not constant anymore.

We would like to end the introduction with the following remark: in this article 
we use the precise correspondence between Ruelle and Laplace resonant states in order to prove the first version of quantum ergodicity for Ruelle resonant states. If however it becomes 
possible to prove stronger properties like quantum unique ergodicity for the high 
frequency limits of Ruelle resonant states, then the quantum-classical correspondence 
would allow to transfer these results to the semiclassical limits of Laplace eigenfunctions. 

\emph{Acknowledgements:} This project has received funding from the European Research Council (ERC) under the European Union’s Horizon 2020 research and innovation programme (grant agreement No. 725967). T.W. is partially supported by Deutsche Forschungsgemeinschaft (DFG) through the Emmy Noether group ``Microlocal Methods for Hyperbolic Dynamics''(Grant No. WE 6173/1-1) . We thank Benjamin K\"uster and Lasse Wolf for helpful comments concerning the manuscript.

\section{Ruelle resonances}\label{sec:PRR}
In this section we introduce the spectral theory of Ruelle resonances for Anosov flows as well as the notion of their resonant states and invariant distributions.

Within this section let $\mathcal M$ be a smooth, compact manifold without boundary, $X\in C^\infty(\mathcal M,T\mathcal M)$  a smooth vector field that generates an Anosov flow $\varphi_t$.

We want to introduce Ruelle resonances as the discrete spectrum of the differential operator $X$ on suitable Hilbert spaces, as they have been introduced by Liverani \cite{Liv04}, Butterley-Liverani \cite{BL07} and Faure-Sj\"ostrand \cite{FSj11} and Dyatlov-Zworski \cite{DZ16}. We will use the microlocal approach from \cite{FSj11, DZ16}. Generalizations to flows on noncompact manifolds can be found in \cite{DG16, BW17}.

\begin{prop}
\label{thm:fredholm_property}
 There exists a family of Hilbert spaces denoted by $\mc{H}^N$ and parametrized by $N>0$. Each $\mc{H}^N$ is an anisotropic Sobolev space that fulfills the relations
\[
H^N(\mc{M})\subset \mc{H}^N(\mathcal M)\subset H^{-N}(\mathcal M),\]
where $H^N(\mathcal M)$ denotes the ordinary $L^2$-based Sobolev space of order $N$. Consider furthermore the operator $X$ acting on $\mathcal D'(\mathcal M)$ and define
\[
{\rm Dom}_{N}(X):=\{u\in \mc{H}^N\mid  X u\in \mc{H}^N\},
\]
which is a dense subset in $\mc{H}^N$. Then the operator $X: \mc{H}^N\to \mc{H}^N$ is an unbounded closed operator defined on a dense domain satisfying:
\begin{enumerate}
 \item There is $C_0>0$ such that for any $N>0$ the operator $(X+\lambda):{\rm Dom}_{N}(X)\to \mc{H}^N$ is Fredholm of index $0$ depending analytically on $\lambda$ in the region $\{\tu{Re}(\lambda) >-N/C_0\}$.
\item There is a constant $C_1>0$ such that for $\Re\lambda>C_1$ the operator $(X+\lambda):{\rm Dom}_{N}(X)\to \mc{H}^N$ is invertible for all $N>0$.
\end{enumerate}
\end{prop}

\begin{proof}
For the Fredholm property see \cite[Theorem 1.4]{FSj11} (\cite[Proposition 3.2]{DZ16}, respectively) and for the invertibility \cite[Lemma 3.3]{FSj11} (\cite[Proposition 3.1]{DZ16},  respectively).
\end{proof}

Consequently, the operator $-X$ has discrete spectrum
on $\{\Re(\lambda)>-N/C_0\}$ of finite algebraic multiplicity. We recall that the algebraic multiplicity of an eigenvalue $\la_0$ of $-X$ is 
$\dim \{u\in \mc{H}^N\mid \exists j\geq 1, (X+\la_0)^ju=0\}$, while the geometric multiplicity of $\la_0$ is 
$\dim \ker_{\mc{H}^N}(X+\la_0)$. 
We call $\lambda \in \C$ a \emph{Ruelle resonance} if
\[
\tu{Res}_{X}(\lambda) := \ker_{\mc{H}^N}(X+\lambda) \neq 0
\]
for some $N> -C_0\Re(\lambda)$.

It can also be shown \cite[Theorem 1.5]{FSj11} that for each $j\geq 0$, $\ker_{\mc{H}^N}(X+\lambda)^j\subset \mathcal D'(\mathcal M)$ is independent of the choice of $N> -C_0\cdot \Re(\lambda)$, so $\tu{Res}_{X}(\lambda)$ is well defined and we call it the space of \emph{Ruelle resonant states} associated to the resonance $\lambda$. In general the geometric and algebraic multiplicity of a Ruelle resonance $\lambda$ need not be equal. Thus we define $J(\lambda)$ to be the smallest integer such that
\[
 \ker_{\mc{H}^N}(X+\lambda)^k = \ker_{\mc{H}^N}(X+\lambda)^{J(\lambda)}
\]
for all $k\geq J(\lambda)$. We call $\ker_{\mc{H}^N}(X+\lambda)^{J(\lambda)}$
the space of \emph{generalized Ruelle resonant states}. Spectral theory also provides us with a finite rank spectral projector
\[
\Pi_{\lambda_0} : \mc{H}^N \to \mc{H}^N
\]
satisfying $\Pi_{\lambda_0}^2= \Pi_{\lambda_0}$ and $\tu{Ran}(\Pi_{\lambda_0})=\ker_{\mc{H}^N}(X+\lambda_0)^{J(\lambda_0)}$, which commutes with the Anosov flow, i.e. $[ X,\Pi_{\lambda_0}] = 0$. Note that this spectral projector coincides with the residue of the meromorphically continued resolvent as defined in the introduction (see \cite[Section 4]{DZ16}).
For $f\in C^\infty(\mathcal M)$, the multiplication operator by $f$ is continuous on $\mc{H}^N$ and the spectral projector allows to define a distribution
\begin{equation}\label{eq:inv_ruelle_dist}
 \mathcal T_{\lambda_0}:\Abb{C^\infty (\mc{M}) }{\C}{f}{\frac{1}{m_{\lambda_0}}\Tr_{\mc{H}^N}(f\Pi_{\lambda_0})}
\end{equation}
where $m_{\lambda_0} := \tu{dim}(\ker_{\mc{H}^N}(X+\lambda_0)^{J(\lambda_0)})$ is called multiplicity of the resonance $\la_0$. From the invariance of $\Pi_{\lambda_0}$ under the Anosov flow it directly follows that $\mathcal T_{\lambda_0}$ is flow-invariant as well. Note that from the microlocal description
of $\Pi_{\lambda_0}$ in \cite{DZ16} it follows that
$\mathcal T_{\lambda_0} \in \mc{D}'(\mathcal M)$ does not depend on the choice of $\mc{H}^N$ and is an intrinsic invariant distribution associated to each Ruelle resonance which we will call \emph{invariant Ruelle distribution}. Note that if the space of generalized resonant states for $\lambda_0=0$ is one-dimensional, then the invariant Ruelle distribution  corresponds to the 
unique SRB-measure.

We will not need any detailed knowledge on the construction of the Hilbert space structure of $\mc{H}^N$. However, we will use the microlocal description of the resonant states using the wavefront set. Therefore, let 
\[T_m \mathcal M = E_0(m)\oplus E_s(m)\oplus E_u(m)\] 
be the Anosov splitting of the tangent bundle into neutral, stable and unstable bundles. We can introduce the following dual splitting of the cotangent space 
\[T^*_m\mathcal M = E_0^*(m)\oplus E_s^*(m)\oplus E_u^*(m),\]
which is defined by
\[E^*_0(m)(E_s(m)\oplus E_u(m))=0,\,\, 
E^*_u(m)(E_0(m)\oplus E_u(m))=0,\,\, E^*_s(m)(E_0(m)\oplus E_s(m))=0.\]

\begin{lem}[{\cite[Lemma 5.1]{DFG15}}]
\label{lem:microlocal_res_state}
 The space of Ruelle resonant states for a resonance $\lambda_0$  is given by
 \begin{equation}\label{eq:microlocal_res_state}
  \tu{Res}_{X}(\lambda_0) = \{u\in \mathcal D'(\mathcal M)\mid(X+\lambda_0)u=0 
  \tu{ and }\WF(u)\subset E^*_u \},
 \end{equation}
where ${\rm WF}(u)\subset T^*\mc{M}$ denotes the wavefront set of the distribution $u$.  The generalized resonant states can be characterized similarly: 
 \begin{equation}\label{eq:microlocal_gen_res_state}
  \ker(X+\lambda_0)^{J(\lambda_0)} = \{u\in \mathcal D'(\mathcal M)\mid(X+\lambda_0)^{J(\lambda_0)}u=0 \tu{ and }\WF(u)\subset E^*_u \}.
 \end{equation}
\end{lem}

By duality one can define the \emph{co-resonant states} which we will denote by $\tu{Res}_{X^*}(\lambda_0)$ and if the Anosov flow preserves a smooth volume density\footnote{For the geodesic flows considered in the sequel, there is a canonical smooth preserved measure, the so called Liouville measure.}, they can microlocally be described as
\begin{equation}\label{eq:microlocal_cores_state}
  \tu{Res}_{X^*}(\lambda_0) = \{u\in \mathcal D'(\mathcal M)\mid (X- \lambda_0)u=0 \tu{ and }\WF(u)\subset E^*_s \}.
\end{equation}
Note that given a Ruelle resonance $\lambda_0$, since $E_s^*\cap E_u^*=0$, the resonant and co-resonant states satisfy conditions on their wavefront sets ensuring that their product is well-defined in $\mc{D}'(\mc{M})$ by \cite[Theorem 8.2.10]{HoeI}. For $u\in \tu{Res}_{X}(\lambda_0), v\in \tu{Res}_{X^*}(\lambda_0)$, the product is a flow-invariant distribution:
\[
X(u\cdot v) = (X u)\cdot v + u\cdot X v = (\lambda_0-\lambda_0)(u\cdot v)=0.
\]
We will see in Section~\ref{PatSul} that the Patterson-Sullivan distributions on compact locally symmetric spaces of rank one can be interpreted as such a product of resonant states. It turns out that in the case of a nondegenerate Ruelle resonance without Jordan block, (i.e. for $\dim(\tu{Res}_{X}(\lambda_0)) = 1$ and $J(\lambda_0)=1$) also the invariant Ruelle distribution $\mathcal T_{\lambda_0}$ can be expressed in this way.

\section{Ruelle resonances on rank one locally symmetric spaces}
In this section we want to relate the Ruelle resonant states of the so called ``first band'' on rank one locally symmetric spaces to certain distributional vectors in principal series representations (Proposition~\ref{prop:ruelle_gamma_inv_distr}). 

\subsection{Riemannian symmetric spaces}\label{sec:rieman_sym_space}

We first recall some standard notations for Riemannian symmetric spaces. Let $G$ be a noncompact, connected, real, semisimple Lie group of real rank $1$ with finite center and $K\subset G$ a maximal compact subgroup. We will write $G=KAN$ for an Iwasawa decomposition and let $M$ be the centralizer of $A$ in $K$ in what follows.
The Killing form $\mc{K}:\mathfrak{g}\x \mathfrak{g}\to \rr$ is a non-degenerate bilinear form and the Cartan involution $\theta:\mathfrak{g}\to \mathfrak{g}$ on the Lie algebra $\mathfrak{g}$ of $G$ allows to define a natural positive-definite scalar product 
$\cjg \cdot,\cdot\cjd_{\mathfrak{g}}=-\mc{K}(\cdot,\theta\cdot)$ on $\mathfrak{g}$ (and thus on $\mathfrak g^*$ as well). Moreover, as $G/K$ is of rank $1$, i.e. $\dim_\R(A)=1$, we have an isomorphism $\mathfrak a_\C^*\to \cc$ by identifying $\lambda$ with $\lambda(H_0)$, after choosing a suitable element $H_0\in\mathfrak a$: we choose $H_0$ to be the uniquely determined element of $\mathfrak a$ which satisfies $\alpha_0(H_0)=\sqrt{\cjg \alpha_0,\alpha_0\cjd_{\mathfrak{g}}}=:||\alpha_0||$, where $\alpha_0\in\mathfrak a^*$ is the unique simple positive restricted root. We shall denote by $\rho\in\mathfrak a^*$ the half-sum of the positive restricted roots weighted by multiplicity, and let $m_{\alpha_0} := \dim \g_{\pm\alpha_0}$ and $m_{2\alpha_0}:=\dim\g_{\pm2\alpha_0}$ be the multiplicities of the possible restricted roots. In particular, one gets 
$\rho=\demi ||\alpha_0||m_{\alpha_0}+||\alpha_0||m_{2\alpha_0}$ when identifying $\mathfrak a_\cc^*\simeq \cc$.

Under the above assumptions $G/K$ is a Riemannian symmetric space of rank $1$. More precisely, $G/K$ is a hyperbolic space ${\bf H}_{\mathbb{K}}$ where $\mathbb{K}$ is either $\rr,\cc,\hh,\mathbb{O}$ ($\hh$ denotes the quaternions and $\mathbb{O}$ the octonions) and we will denote by $n$ its real dimension\footnote{Note that this choice implies, that e.g. for complex hyperbolic spaces $n$ has always to be chosen even. For $\mathbb{K}=\mathbb{O}$ only $n=16$ is possible (Cayley plane).}. To simplify notation, we will write ${\bf H}=G/K$. The Killing form induces a canonical Riemannian metric on ${\bf H}$ and with this metric the space $\mathbf H$ has negative  sectional curvature (equal to $-1$ if $\mathbb{K}=\rr$ and in $[-4,-1]$ in the other cases). Furthermore, it induces a smooth left $G$-invariant measure, which we denote by $\intd x$. The unit sphere bundle $S{\bf H}$ can be identified with $G/M$ and we denote the left $G$-invariant Liouville measure by $\intd\mu_{\tu{L}}$. Using a trivialization $S{\bf H} \cong {\bf H}\times \mathbb S^{n-1}$ the Liouville measure can be written as $\intd \mu_{\tu{L}} = \intd x\otimes\intd \mu_{\mathbb S^{n-1}}$ where $\intd \mu_{\mathbb S^{n-1}}$ is the standard Lebesgue measure on the unit sphere. Note that the measures $\intd x$ and $\intd \mu_{\tu{L}}$ are intrinsically defined by the Riemannian geometry of ${\bf H}$. We normalize the bi-invariant Haar measures on the Lie groups $G,M,K,A$ and $N$ in a consistent way as follows: we start by fixing $\intd m$ by the condition $\tu{vol}(M)=1$ and in addition set 
$\tu{vol}(K) = \tu{vol}(\mathbb S^ {n-1})$. The adjoint action of $K$ on $\p$ gives an identification $K/M
\cong \mathbb S^{n-1}$ and our choice implies that, under this identification, $\intd \mu_{\mathbb S^ {n-1}} = \intd (kM)$. Next we fix $\intd g$ such that $\intd (gK) = \intd x$. With these choices one obtains the identification $\intd (gM) = \intd \mu_{\tu{L}}$ in the following way: the $G=NAK$ decomposition gives a trivialization $G/M\cong G/K\times K/M$ and using the normalizations from above one checks that $\intd(gM) = \intd (gK)\otimes \intd (kM) = \intd x\otimes\intd \mu_{\mathbb S^{n-1}} = \intd \mu_{\tu{L}}$. Finally it remains to normalize $\intd a$: recall that we identify $\a\cong \R$ or respectively $\a^*_\C\cong \C$ by the choice of an element $H_0$ that is normalized w.r.t. the Killing form. This imposes an analogous normalization of the measure $\intd a$.

Let $\Gamma\subset G$ be a torsion-free discrete co-compact subgroup, then $\textbf{M}:=\Gamma\backslash G/K$ is a smooth compact Riemannian locally symmetric space of rank $1$. We denote the respective positive Laplacians by $\Delta_{\mathbf{H}}$ and $\Delta_{\textbf{M}}$. Again we have a Lebesgue measure defined by the Riemannian metric as well as the Liouville measure and by slight abuse of notation we also denote them by $\intd x$ and $\intd \mu_{\tu{L}}$. 
The unit tangent bundle $\mc{M}:=S{\bf M}$ of ${\bf M}=\Gamma\backslash {\bf H}$ can be identified with the quotient $\Gamma\backslash G/M$. Under this identification the geodesic flow is simply the  right action of $A$: i.e 
\[A \x \Gamma\backslash G/M\to \Gamma\backslash G/M, \quad (a, \Gamma gM)\mapsto \Gamma gaM.\] 
It is known to be an Anosov flow (see e.g. \cite{Hi05}), thus all the definitions
of Ruelle resonances and resonant states from Section~\ref{sec:PRR} apply. Moreover, the Anosov splitting of $T\mc{M}$  into neutral, stable, and unstable directions can be expressed explicitly as associated vector bundles 
$T\mc{M}=\Gamma\backslash G\times_M(\mathfrak a\oplus\mathfrak n_+\oplus  \mathfrak n_-)$. %\cite[\S~2.3]{Ju01}. 
Here  $\mathfrak n_+$ is the Lie algebra of $N$ and $\mathfrak n_-=\theta\mathfrak n_+$ is the Lie algebra of 
$\overline{N}:=\theta N$, where $\theta$ denotes the Cartan involution on $G$ as well as its derivative on the Lie algebra $\mathfrak g$ of $G$. The bundles $E_u$ and $E_s$ are identified with
\[E_u=\Gamma\backslash G\times_M \mathfrak n_-, \quad E_s=\Gamma\backslash G\times_M \mathfrak n_+.\] 
Similarly, the geodesic flow on the cover $S{\bf H}=G/M$ also has an Anosov splitting with smooth stable and unstable bundles  
\[ \til{E}_u:=G\times_M \mathfrak n_-, \quad \til{E}_s:=G\times_M \mathfrak n_+.\] 

\subsection{The first band of classical Ruelle resonances}\label{sec:firstband}
In \cite{DFG15} it is shown for compact real hyperbolic manifolds that the spectrum of Ruelle resonances forms an exact band structure. A particularly important subset
of resonances are those that are invariant by the horocyclic flows (i.e. killed by the unstable derivatives).

\begin{Def}\label{def:firstBand}
A Ruelle resonant state $u$ is said to belong to the 
\emph{first band}, if for each smooth section $U_-$ of $E_u$ 
we have $U_- u =0$ and we write ${\rm Res}_X^0(\lambda_0)$ 
for the first band Ruelle resonant states at the resonance 
$\lambda_0\in \C$.  

Similarly, we say a co-resonant state $u$ belongs to the 
\emph{first band} if for each section $U_+$ of $E_s$ we 
have $U_+u =0$. We write $ \tu{Res}^0_{X^*}(\lambda_0)$ 
for the first band Ruelle co-resonant states at the 
resonance $\lambda_0\in \C$. 
\end{Def} 
\begin{rem}
 The notion \emph{first band} is justified by the following result
 of an exact band structure (see \cite{DFG15} for constant negative 
 curvature manifolds and \cite{KW19a} for compact locally symmetric spaces
 of rank one): If $\lambda_0\in \C$ is a Ruelle resonance with 
 $\tu{Im}(\lambda_0)\neq 0$, then 
 $\tu{Re}(\lambda_0) \in -\rho - \N_0\alpha_0$, i.e. the resonances 
 with nonvanishing imaginary parts are arranged on vertical lines parallel 
 to the imaginary axis. Furthermore, if $\lambda= -\rho + ir$ for $r\neq 0$
 i.e. if $\lambda$ lies on the first line, then it is shown in \cite{DFG15,KW19a}, that the resonant states are first band resonances in the sense of Definition~\ref{def:firstBand}.
\end{rem}

We can lift the resonant states to the cover $G/M$ by the quotient map $\pi_{\Gamma}\colon G/M\to \Gamma\backslash G/M$. By a slight abuse of notation, $X$ will also denote the infinitesimal generator of the geodesic flow on $G/M$ which descends to the infinitesimal generator of the geodesic flow on $\Gamma\backslash G/M$ via $\pi_{\Gamma}$. The geodesic flow (the flow of $X$) on $\mc{M}=\Gamma\backslash G/M$ and on $S{\bf H}=G/M$ will be denoted by $\varphi_t$. The splitting 
$G\times_M( \mathfrak a +\mathfrak n_+ +\mathfrak n_-)$ of the tangent bundle of $G/M$ is $G$-invariant and descends to the Anosov splitting for 
$\Gamma\backslash G/M$ via $\pi_{\Gamma}$. 
With this notation, for $\lambda\in\a^*_\C$, let 
\begin{eqnarray*}
  \mc{R}_\pm(\lambda) 
 =\{u\in \mathcal D'(G/M)\mid (X\mp \lambda(H_0))u=0,
 \forall U_\pm \in C^\infty(G/M;G\times_M \mathfrak n_\pm),  U_\pm u =0\}.
\end{eqnarray*}

\begin{rem}\label{lem:firs_band_Gamma_inv_vectors}
  In view of the above characterizations of $\tu{Res}^0_{X}(\lambda)$   and $\tu{Res}^0_{X^*}(\lambda)$ we obtain the linear isomorphisms
 \[
  (\pi_{\Gamma})^*:\tu{Res}_{X}^0(\lambda(H_0)) \to {^\Gamma \mc{R}_-(\lambda)}, \quad
  (\pi_{\Gamma})^*:\tu{Res}_{X^*}^0(\lambda(H_0)) \to {^\Gamma \mc{R}_+(\lambda)},
 \]
where $^\Gamma \mc{R}_\pm(\lambda)$ denotes the subspace of $\Gamma$-invariant elements of $\mc{R}_\pm(\lambda)$.
\end{rem}

\subsection{Points at infinity} 
Next we want to identify the first band of (co)-resonant states with distributions on the Furstenberg 
boundary $\pl{\bf H}$ of the symmetric space, which is identified with $G/P=K/M$ where $P=MAN$ is the minimal parabolic of $G$. Let us explain how this boundary can be naturally obtained from the geodesic flow on 
$G/M$. For any point $y\in S{\bf H}=G/M$ we define the limiting points of the geodesic passing through $y$:
\[ B_\pm(y):=\lim_{t\to +\infty}\til\pi_0(\varphi_{\pm t}(y))\in \pl{\bf H} \] 
if $\til\pi_0:S{\bf H}\to {\bf H}$ is the projection onto the base. In terms of Lie groups, the resulting maps are simply the projections 
\begin{equation}\label{eq:initial_end_point_map}
B_\pm\colon  G/M \to G/P=K/M,\quad gM\mapsto gw_\pm P, 
\end{equation}
where $w_+$ represents the trivial and $w_-$ the nontrivial element of the Weyl group $W=N_K(A)/M\sim \Z_2$    (cf. \cite{DFG15,Hi05,HHS12}); here $N_K(A)$ denotes the normalizer of $A$ in $K$.  We will refer to $B_-$ and $B_+$ as the \emph{initial} respectively the \emph{end point map}. 
 
\begin{rem}\label{lem:projection_properties}
The values of the initial and the end point map are invariant under the geodesic
 flow on $G/M$, i.e.
 \begin{equation}
  \label{eq:pr_flow_invariance}
   \forall a\in A, \quad B_{\pm}(gaM)=B_{\pm}(gM) .
 \end{equation}
  Furthermore, the initial and the end point map are invariant under changes in the unstable, resp. the stable direction.
 \begin{equation}
  \label{eq:pr_stable_invariance}
   \forall U_\pm\in \n_\pm :\quad \frac{d}{ds}_{|s=0}B_\pm (ge^{s U_\pm}M)=0.
 \end{equation}
 To see \eqref{eq:pr_flow_invariance} and \eqref{eq:pr_stable_invariance} for the initial point map, note that $\Ad(w_-)$ interchanges $\n_+$ and $\n_-$  and is $-\mathrm{id}$ on $\mathfrak a$. Finally, both maps intertwine the left $G$-actions on $G/M$ and $K/M=G/P$
 \begin{equation}
  \label{eq:pr_intertwining}
  \forall \gamma\in G~:~B_\pm(\gamma gM) = \gamma\cdot B_\pm(gM)
 \end{equation}
\end{rem} 

\begin{lem}\label{lem:Q_pm}
The  maps
\[
 \mathcal Q_\pm: \Abb{\mathcal D'(\pl{\bf H})}{\mc{R}_\pm(0)\subset \mc{D}'(S{\bf H})}{T}{B_{\pm}^*T}
 \]
 are topological, linear isomorphisms intertwining the pullback actions of $G$  on $S{\bf H}=G/M$ and $\pl{\bf H}=K/M$.
\end{lem}

\begin{proof}
 As $B_{\pm}:G/M\to K/M$ are surjective submersions, the pullback operators $B_{\pm}^*$ are injective operators on  $\mathcal D'(K/M)$. We show that their images lie in the indicated spaces: for $T\in \mathcal D'(K/M)$, Remark~\ref{lem:projection_properties} implies
\[
X (B^*_{\pm}T)=0 \quad \textrm{ and } U_\pm (B^*_\pm T)=0
\]
for all  $U_\pm\in C^\infty(G/M;G\times_M \n_\pm)$.
For the surjectivity of $\mathcal Q_+$ consider a $v\in \mathcal D'(G/M)$ such that
\begin{equation}
 \label{eq:flow_and_instable_invariance}
X v= 0\tu{ and }~\forall U_+ \in C^\infty(G/M; G\times_M \n_+),~ 
U_+ v= 0~.
\end{equation}
Now consider the $AN$-fiber bundle $B_+: G/M \to G/P$  and note that \eqref{eq:flow_and_instable_invariance} 
implies that $v$ is a distribution which is constant along the fibers. Thus there is a $T_+\in\mathcal D'(G/P)$ such that $v= B_+^*T_+ $. Analogously, we obtain the surjectivity of $\mathcal Q_-$.

The fact that $\mathcal Q_\pm$ intertwines the pullback action is a direct consequence of the intertwining property \eqref{eq:pr_intertwining} of $B_\pm$.

The continuity of $\mathcal Q_\pm$ follows directly from the continuity of push-forwards of distributions under submersions. For the continuity of the inverse map consider the embedding $\iota: K/M \hookrightarrow G/M$.
Then \eqref{eq:flow_and_instable_invariance} implies that for elements $v\in \mc{R}_+(0)$, $\WF(v)$ is a subset of the annihilator of $G\x_M(\a \oplus\n_+)$ in $T^*(G/M)$. 
Thus the Hörmander condition for pullbacks \cite[Thm 8.2.4]{HoeI} implies that $\iota^*:\mc{R}_\pm(0) \to \mathcal D'(K/M)$ is well defined and continuous.
As $B_\pm\circ \iota = \tu{Id}_{K/M}$ and $\mathcal Q_\pm$ is bijective, we find that $\iota^*=\mathcal Q_\pm^{-1}$ and we deduce the continuity of the inverse map.
\end{proof}

Similar isomorphisms can also be defined for nontrivial resonances $\lambda\neq 0$. This requires the so called \emph{Busemann function} (based at the origin) defined on ${\bf H}\x \pl{\bf H}=G/K\times K/M$ by
\begin{equation}
 \label{eq:def_horocycle}
 \beta: \Abb{G/K\times K/M}{\a}{(gK,kM)}{\beta(gK,kM) :=-H(g^{-1}k)},
\end{equation}
where $H\colon G\to \mathfrak a$ is given by $H(kan)=\log a$. Using this horocycle bracket as well as the initial and end point maps, we can define the smooth functions on 
\begin{equation}
 \label{eq:def_Phi_pm}
  \Phi_{\pm}:\Abb{G/M}{\R}{gM}{e^{\nu_0 \beta(gK,B_\pm(gM))}},
 \end{equation}
where $\nu_0:=\alpha_0/\|\alpha_0\|$. A straightforward computation using (\ref{eq:pr_flow_invariance}) and (\ref{eq:pr_stable_invariance}) shows that they are $\pm 1 $ eigenfunctions of the geodesic vector field,
\begin{equation}
\label{eq:Phi_pm_flow}
X\Phi_\pm = \frac{d}{dt}_{|t=0} \Phi_\pm (ge^{tH_0} M)  = \pm\Phi_\pm,
\end{equation}
and they are constant in the stable and unstable directions, respectively:
\begin{equation}
\label{eq:Phi_pm_stab_unstab}
  \forall U_\pm\in \n_\pm :\quad U_\pm\Phi_\pm= \frac{d}{ds}_{|s=0} \Phi_{\pm}(ge^{sU_\pm} M) =0.
\end{equation}

For $\gamma\in G$, $x\in {\bf H}=G/K$ and $b\in \pl{\bf H}= K/M$  we have the following equalities
\begin{eqnarray}\label{eq:horo_group_action1}
 \beta (\gamma x,\gamma b) &=& \beta(x,b) + \beta(\gamma K,\gamma b)\\
\label{eq:horo_group_action2}
\beta(\gamma K,\gamma b) &=& -\beta(\gamma^{-1}K,b),
\end{eqnarray}
see for example \cite[Lemma 2.3]{HHS12}. Then for $\gamma, g \in G$ we obtain 
\begin{eqnarray*}
 \Phi_\pm(\gamma gM) &\underset{(\ref{eq:pr_intertwining})}=&
 e^{\nu_0\beta( \gamma gK,\gamma.B_\pm(gM))}
 \underset{(\ref{eq:horo_group_action1})}{=} e^{\nu_0\beta( gK,B_\pm(gM))} e^{\nu_0\beta(\gamma K, \gamma B_\pm(gM))} \\
 &\underset{(\ref{eq:horo_group_action2})}{=}&e^{\nu_0\beta(gK,B_\pm(gM))} e^{-\nu_0\beta( \gamma^{-1} K, B_\pm(gM))}.
\end{eqnarray*}
If we introduce, generalizing \cite[(3.26)]{DFG15}, the function
\[
 N_\gamma(kM) := e^{-\nu_0\beta(\gamma^{-1} K, kM)}
\]
 on $\pl{\bf H}=K/M$, we get the identity (cf. \cite[(3.28)]{DFG15})
\begin{equation}
 \label{eq:Phi_pm_trafo}
 \Phi_\pm(\gamma gM) = N_\gamma(B_\pm(gM))\Phi_\pm( gM).
\end{equation}

A last ingredient is the compact picture of the spherical principal series.

\begin{Def}\label{prop:compact_pic}
Let $\mu\in\a^*_\C$ and $H_\mu:= L^2(\pl{\bf H})$ be the Hilbert space of square integrable  functions w.r.t. the $K$-invariant  measure on $\pl{\bf H}=K/M$. Then the \emph{spherical principal  series representation} $(\pi^{\mathrm{cpt}}_\mu, H_\mu)$ is the representation of $G$ on  $H_\mu$ given by 
\begin{eqnarray}\label{eq:compact_picture}
 (\pi^{\mathrm{cpt}}_\mu(\gamma) f) (k'M) 
 &:=& e^{-(\mu+\rho)H(\gamma^{-1}k')}f(k(\gamma^{-1}k')M)\\
 &=&\nonumber
 e^{(\mu+\rho)\beta(\gamma K,k'M)}f(k(\gamma^{-1}k')M).
\end{eqnarray}
Here  $k(\gamma^{-1}k')=k_{KAN}(\gamma^{-1}k')$ denotes the $K$-component of $\gamma^{-1}k'\in G$ in the $KAN$-decomposition.
\end{Def}
Note that these principal series representations are unitary iff $\mu\in i\R$.
Furthermore we can express (\ref{eq:compact_picture}) via the functions $N_\gamma$ as follows:
\begin{equation}
 \label{eq:compact_picture_N_gamma}
 (\pi^{\mathrm{cpt}}_\mu(\gamma) f)= N_{\gamma^{-1}}^{-(\mu+\rho)(H_0)}(\gamma^{-1})^* f,
\end{equation}
where $(\gamma^{-1})^*$ is the pullback of distributions with the diffeomorphism
obtained by the left $G$-action on $K/M=G/P$.

\begin{prop}\label{prop:classical_compact_intertwining}
 For $\lambda \in \a^*_\C$ the initial and end point transformations  defined by
\begin{equation}
 \label{eq:Q_lambda_def}
  \mathcal Q_{\lambda,\pm}(T) := \Phi_\pm^{\lambda(H_0)}\mathcal Q_\pm(T)
\end{equation}
 are 
topological isomorphisms
\[
  \mathcal Q_{\lambda,\pm}: \mathcal D'(\pl{\bf H})\to \mc{R}_\pm(\lambda)
\]
intertwining the left regular representation on
$\mc{R}_\pm(\lambda) \subset \mc{D}'(S{\bf H})$ with the representation 
$\big(\pi^{\mathrm{cpt}}_{-(\lambda+\rho)},\mathcal D'(\pl{\bf H})\big)$.
\end{prop}
\begin{proof} In view of Lemma~\ref{lem:Q_pm} and the properties of $\Phi_\pm$, the property of 
$\mathcal Q_{\lambda,\pm}$ being a topological isomorphism is clear. 
It only remains to verify the intertwining property: 
for $\gamma \in G$ and $T\in\mathcal D'(K/M)$
we compute
\begin{eqnarray*}
 (\gamma^{-1})^*(\mathcal Q_{\lambda,\pm}T) &=&\Big((\gamma^{-1})^* \Phi_\pm^{\lambda(H_0)}\Big)
                                               \cdot\Big((\gamma^{-1})^* B_\pm^*T\Big)\\
 &\underset{(\ref{eq:Phi_pm_trafo}),(\ref{eq:pr_intertwining})}{=}&
         \Big(B_\pm^*(N_{\gamma^{-1}}^{\lambda(H_0)}) \Phi_\pm^{\lambda(H_0)}\Big)\cdot
         \Big( B_\pm^*(\gamma^{-1})^*T\Big)\\
  &=&\Big(\Phi_\pm^{\lambda(H_0)}\Big)\cdot
  \Big( B_\pm^*\left(N_{\gamma^{-1}}^{\lambda(H_0)}\cdot (\gamma^{-1})^* T\right)\Big)\\
  &\underset{(\ref{eq:compact_picture_N_gamma}), \lambda= -(\mu+\rho)}{=}&\Big(\Phi_\pm^{\lambda(H_0)}\Big)\cdot
  \Big( \mathcal Q_\pm(\pi_\mu^{\mathrm{cpt}}(\gamma) T)\Big)\\
  &= & \mathcal Q_{\lambda,\pm}(\pi^{\mathrm{cpt}}_\mu(\gamma)T).
\end{eqnarray*}
\end{proof}

Combining Proposition~\ref{prop:classical_compact_intertwining} with Remark~\ref{lem:firs_band_Gamma_inv_vectors} we arrive at the promised description of the first band of Ruelle resonances.

\begin{prop} \label{prop:ruelle_gamma_inv_distr}
 There are isomorphisms of finite dimensional vector spaces
 \[
  \tu{Res}_{X}^0(\lambda) \cong {^\Gamma(H_{-(\lambda+\rho)}^{-\infty})}
 \quad
\text{and}
\quad
  \tu{Res}_{X^*}^0(\lambda) \cong {^\Gamma(H_{-(\lambda+\rho)}^{-\infty})},
\]
where ${^\Gamma(H_{-(\lambda+\rho)}^{-\infty})}$ denotes the spaces of $\Gamma$-invariant distributional vectors in the spherical principal series with spectral parameter $\mu=-(\lambda+\rho)$.
\end{prop}

After having described the Ruelle resonances by distributions on the
boundary, we now turn to the description of generalized resonant states
via boundary distributions.
\begin{prop}\label{prop:ruell_jordan_bv}
 Let $\lambda\in \C$ be a Ruelle resonance of $X$ on $\mc{M}=\Gamma\backslash G/M$. Then the following conditions are equivalent.
 \begin{itemize}
 \item[(1)]
 There is a Jordan block of first band resonant states of  size $J$, 
 i.e. there are distributions $u_0,\ldots, u_{J-1} \in \mc{D}'(\mc{M})$ 
 and some $\la$ such that
  \begin{equation}\label{relationJB}
  (X+\lambda)u_0 = 0 \tu{ and } (X + \lambda)u_k = u_{k-1} \tu{ for } k=1,\ldots,J-1
 \end{equation}
 and $U_- u_k=0$ for all smooth sections $U_-$ of the bundle $E_u=\Gamma\backslash G\times_M \n_-$.

\item[(2)] There exist distributions $T_0,\ldots, T_{J-1}\in \mathcal D'(\pl{\bf H})$
 such that for $1\leq k\leq J-1$ and all $\gamma\in \Gamma$
 \begin{equation}\label{eq:T_k_transformation}
  \gamma^* T_k = N_\gamma^{-\lambda}\sum_{l=0}^k\frac{(\log N_\gamma)^{k-l}}{(k-l)!}T_{l}
 \end{equation}
 \end{itemize}
\end{prop}

\begin{proof}
 We start with a Jordan basis $u_k$  on $\mc{M}=\Gamma\backslash G/M$
 as in (1) and we lift them to a set of $\Gamma$-invariant distributions on the cover $S{\bf H}=G/M$
\begin{equation}
 \label{eq:jordan_relation}
 \til u_k := (\pi_{\Gamma})^*u_k, ~~k= 0,\ldots, J-1.
\end{equation}
where
 $\pi_{\Gamma} : S{\bf H}\to \mc{M}= S{\bf M}$ is the natural projection. They satisfy 
 the relations \eqref{relationJB} with $\til{u}_k$ replacing $u_k$.
Then we use the functions $\Phi_-$ defined in (\ref{eq:def_Phi_pm}) and define the distributions on $S{\bf H}$
 \begin{equation}
  \label{eq:v_k_u_k}
  v_k := \Phi_-^{-\lambda} \sum_{l=0}^k \frac{(\log \Phi_-)^{k-l}}{(k-l)!}\til u_l.
 \end{equation}
 From  (\ref{eq:Phi_pm_stab_unstab})
 we deduce $U_- v_k =0$ for all sections $U_-$ in $\til{E}_u$. 
 Using (\ref{eq:Phi_pm_flow}) as well as the $\Gamma$-invariance of
 $\til u_k$, a straightforward computation yields $X v_i =0$. According to 
 Lemma~\ref{lem:Q_pm} there are unique distributions $T_k\in \mathcal D'(\pl{\bf H})$ fulfilling
 $v_k = \mathcal Q_- T_k$. The transformation property (\ref{eq:T_k_transformation})
 finally follows, as $\mathcal Q_-$ intertwines the pullback actions, from the following computation for $\gamma\in \Gamma$:
 \begin{eqnarray*}
  \gamma^*v_k 
  &=& \gamma^*\Phi_-^{-\lambda}\sum_{l=0}^k \frac{(\log \gamma^*\Phi_-)^{k-l}}{(k-l)!}\til u_l\\
  &\underset{(\ref{eq:Phi_pm_trafo})}{=}&(B_-^* N_\gamma^{-\lambda}) \Phi_-^{-\lambda}
    \sum_{l=0}^k \frac{(\log \Phi_-+B_-^*\log N_\gamma)^{k-l}}{(k-l)!}\til u_l\\
    &=& (B_-^* N_\gamma^{-\lambda})\Phi_-^{-\lambda}
    \sum_{l=0}^k \sum_{r=0}^{k-l} \frac{(k-l)!}{r!(k-l-r)!}\frac{(\log \Phi_-)^r(B_-^*\log N_\gamma)^{k-l-r}}{(k-l)!}\til u_l\\
    &=&B_-^*N_\gamma^{-\lambda}\sum_{j=0}^k \frac{(B_-^*\log N_\gamma)^{k-j}}{(k-j)!}v_j
 \end{eqnarray*}
 This proves that (1) implies (2). 
 
The converse follows similarly: Given $T_0,\ldots, T_{J-1} \in \mathcal D'(\pl{\bf H})$ fulfilling 
 (\ref{eq:T_k_transformation}) we can define $v_k = \mathcal Q_- T_k$. Next, we can obtain  $\til u_l$ from the $v_k$ by (\ref{eq:v_k_u_k}). From (\ref{eq:T_k_transformation})  we conclude that $\til u\in \mathcal D'(S{\bf H})$ are $\Gamma$-invariant, thus we  obtain distributions $u_k\in \mathcal D'(\mc{M})$. By a straightforward computation
 they fulfill (\ref{eq:jordan_relation}) and $U_- u_k=0$ for all smooth sections
 $U_-$  of $E_u$. 
\end{proof}

\section{Quantum-classical correspondence}
\subsection{Poisson Transformation}
A central role for the relation between classical and quantum resonances is played by the Poisson transformation which we now introduce. 

As explained in Section \ref{sec:firstband}, we can identify $\mathfrak{a}^*$ with $\cc$
via $\la\mapsto \la(H_0)$, and we shall do so in what follows, writing $\la^2$ instead of $\la(H_0)^2$. 
Given a spectral parameter $\mu\in \mathfrak a_\C^*$ we introduce the eigenspace 
(cf.  \cite{vBS87})
\[
 \mathcal E_\mu({\bf H}) :=\{u\in \mathcal D'({\bf H})\mid (\Delta_{{\bf H}}-\rho^2+\mu^2) u=0\}
\]
for the positive Laplacian $\Delta_{{\bf H}}$ on ${\bf H}$. Note that by elliptic regularity the elements of $\mathcal E_\mu({\bf H})$ are real analytic. If we define the space of \emph{quantum eigenstates} of $\Delta_{\bf M}$ 
on ${\bf M}$ as
\[
 \tu{Eig}_{\Delta_{\bf M}}(\mu)
 :=\{u\in L^2({\bf M})\mid (\Delta_{\bf M} -\rho^2+\mu^2) u=0\},
\]
taking the lift to the universal cover $\til\pi_{\Gamma}: {\bf H} \to {\bf M}$ we obtain a bijection between eigenfunctions of $\Delta_{\bf M}$ and $\Gamma$-invariant elements in $\mathcal E_\mu({\bf H})$, denoted by
\[
  ^\Gamma \mathcal E_\mu({\bf H}):=\{f\in \mathcal E_\mu({\bf H})\mid   \forall \gamma\in \Gamma, \gamma^*f=f\}.
\]

\begin{Def}
 Let $\mu\in\a^*_\C$ and define
 \[
 p_\mu(gK, kM) := e^{(\mu+\rho) \beta(gK,kM)} \in C^\infty(G/K\times K/M),
 \]
which is the Schwartz kernel of the \emph{Poisson transformation} and  which defines,
using the $K$-invariant measure $\intd b:=\intd \mu_{\mathbb S^ {n-1}} = \intd(kM)$ on $K/M$, a linear operator 
 \[
  P_\mu: \Abb{\mathcal D'(K/M)}{C^\infty(G/K)}{T}{T(\intd b)[p_\mu(gK,kM)]}.
 \]
Here we use the notation that $T(\intd b)$ is the generalized density 
associated with the distribution $T$ via the invariant measure $\intd b$.
\end{Def}

In the case of rank $1$ symmetric spaces, kernel and image of the Poisson transformation have very explicit descriptions. In this paper we mostly restrict our attention to spectral parameters $\mu$ for which the Poisson transformation is injective.
The maximal domains of definition for the Poisson 
transformations are spaces of hyperfunctions. As we restrict our attention to spaces of distributions we need to introduce spaces of smooth functions with moderate 
growth in order to describe the image of our Poisson transformations.

For $f\in C^\infty({\bf H})$ and $r\geq 0$ the norm
\[
 \|f\|_r:= \sup_{x\in {\bf H}} \left|f(x) e^{-r\cdot d_{{\bf H}}(o,x)}\right|,
\]
where $d_{{\bf H}}$ is the Riemannian distance function on ${\bf H}=G/K$ and $o=eK$ is the base point of ${\bf H}$. Then we define
\[
 \mathcal E^r_\mu({\bf H}) :=\{f\in \mathcal E_\mu({\bf H}), \|f\|_r\leq\infty\}.
\]
%Together with $\|\cdot \|_r$ these spaces become Banach spaces. 
The space of eigenfunctions of \emph{weak moderate growth} (see \cite[Remark 12.5]{vBS87}) can then be defined as
\begin{align*}
 \mathcal E^*_\mu({\bf H}) := \bigcup_{r>0}\mathcal E^r_\mu({\bf H})
\end{align*}
and we can equip it with the direct limit topology. 

In the following proposition we collect the mapping properties of the Poisson transformation we will use (cf. \cite{vBS87}).

\begin{prop}\label{thm:poisson_trafo}
Define the set of \emph{exceptional parameters} to be
\begin{equation}
 \label{eq:exceptional_points}
 \mathbf{Ex} := \left(-\frac{m_{\alpha_0}}{2} - 1-2\N_0\right)\alpha_0
 \bigcup\left(-\frac{m_{\alpha_0}}{2} - m_{2\alpha_0}-2\N_0\right)\alpha_0\subset \a^*_\C.
\end{equation}
For $\mu\in \a^*_\C$ the Poisson transformation $P_\mu$ is a bounded continuous map

\begin{equation}\label{eq:P_mu_isom_dist}
 P_\mu:\mathcal D'(\pl{\bf H})\to \mathcal E^*_\mu({\bf H}),
\end{equation}
 which is a topological isomorphism if and only if $\mu \notin \mathbf{Ex}$.
 \end{prop}

Note that $\mathcal E^*_\mu({\bf H})$ is invariant under the left regular representation. Moreover, if one considers the compact picture $\pi_{-\mu}^{\mathrm{cpt}}$ of the spherical principal series representation of 
$G$ associated with the spectral parameter $\mu$, then $\mathcal D'(\pl{\bf H})$ can be interpreted as the space of distribution vectors of 
$\pi_{-\mu}^{\mathrm{cpt}}$ (cf. \cite{vBS87}). As is well-known, the Poisson transformation $P_\mu$ intertwines these two representations.

For later reference we collect some of the spectral properties of $\Delta_{\bf M}$.

\begin{rem}\label{lem:lift_quantum}
 Let $\Gamma\subset G$ be a co-compact discrete subgroup. Then for all $\mu\in \a^*_\C$
 the pullback of smooth functions is a bijection
 \[
  (\til\pi_{\Gamma})^*: \tu{Eig}_{\Delta_{\bf M}}(\mu)\to{^\Gamma \mathcal E_\mu({\bf H})},
 \]
 where ${^\Gamma \mathcal E_\mu({\bf H})}$ denote the space of $\Gamma$-invariant  elements in $\mc{E}_\mu({\bf H})$. Consequently $^\Gamma \mathcal E_\mu({\bf H})$ is finite dimensional and $^\Gamma \mathcal E_\mu({\bf H})\neq 0$ only  holds  on a discrete set of values for $\mu\in \a^*_\C$ that fulfill $\|\rho\|^2-\|\mu\|^2\geq 0$. Furthermore,
\[
^\Gamma \mathcal E_\mu({\bf H})\subset \mathcal E_\mu^0({\bf H}) \subset \mathcal E^*_\mu({\bf H}).
\]
\end{rem}

\subsection{Correspondence of spectra and resonant states}
Let us now come to the proof of the first main result. Consider the canonical projection $\til{\pi}_0: S{\bf H}= G/M\to {\bf H}=G/K$ on the base of the fibration 
and similarly $\pi_0: S{\bf M}=\Gamma\backslash G/M\to {\bf M}=\Gamma\backslash G/K$. The pullback of smooth functions leads to a map $\pi_0^*: C_c^\infty({\bf M}) \to C_c^\infty(S{\bf M})$. Using the canonical measures on both spaces we can identify distributions as dual spaces of smooth compactly supported functions and by duality we obtain a map ${\pi_0}_*: \mathcal D'(S{\bf M}) \to \mathcal D'({\bf M})$. In the same way we obtain the associated pushforward on distributions $\til \pi_{0*}$ on the universal cover, which allows us to state the following useful expressions for the Poisson transformation.

\begin{prop}\label{prop:poisson_projection}
Let $\mu\in \a^*_\C$ be an arbitrary spectral parameter. Then we have the equalities $P_\mu = \til{\pi}_{0*} \circ \mathcal Q_{\mu-\rho, -}$, and
$P_\mu = \til{\pi}_{0*} \circ \mathcal Q_{\mu-\rho, +}$, respectively, as maps $\mathcal D'(\pl {\bf H}) \to \mathcal D'({\bf H})$.
\end{prop}

\begin{proof}
 By a density argument we restrict our attention to the map for smooth functions $\phi\in C^\infty(K/M)$. Recall from Section~\ref{sec:rieman_sym_space} that the invariant 
 measures on $G/K$, $G/M$ and $K/M$ are normalized such that
 \[
  \int_{G/M}f(gM) \intd (gM) = \int_{G/K}\int_{K/M}  f(gkM) ~\intd (kM)\intd (gK).
 \]
For $f\in C^\infty_c(G/M)$ and $\phi\in C^\infty_c(G/K)$ we find
\begin{eqnarray*}
 \int_{G/K}\til\pi_{0*}f\cdot \phi ~\intd (gK)
 &=&\int_{G/M}f\cdot \til\pi_{0}^*\phi ~\intd (gM)\\
 &=&\int_{G/K}\int_{K/M}f(gkM)\, \til\pi_{0}^*\phi(gkM)~\intd (kM) \intd (gK)\\
 &=&\int_{G/K}\int_{K/M}f(gkM)\, \phi(gK)~\intd (kM) \intd (gK).
\end{eqnarray*}   
This implies 
\begin{equation}\label{eq:pi0pushforward}
\til\pi_{0*}f(gK)=\int_{K/M}f(gkM)~\intd (kM).
\end{equation}
For $\psi\in C^\infty(K/M)=C^\infty(G/P)$ we recall $\mathcal Q_{\mu-\rho,-}(\psi)\in C^\infty(G/M)$ from \eqref{eq:Q_lambda_def} and note that (we use the isomorphism $gP\mapsto k(g)M$ identifying $G/P$ with $K/M$, $k(g)$ being the $K$ element in the $KAN$-decomposition of $g$)
\begin{eqnarray*}
(\mathcal Q_{\mu-\rho,-}\psi)(gM)
&=&\Phi_-^{(\mu-\rho)(H_0)}(gM)\psi(gw_-P)\\
&=&(e^{\nu_0\beta( gK,B_-(gM))})^{(\mu-\rho)(H_0)}\psi(gw_-P)\\
&=&e^{(\mu-\rho)\beta( gK, k(gw_-)M)}\psi(k(gw_-)M).
\end{eqnarray*}

Now \eqref{eq:pi0pushforward} and \cite[Lemma I.5.19]{Hel84} allow us to compute
 \begin{align*}
[(\til{\pi}_{0*} \circ \mathcal Q_{\mu-\rho, -}) \psi](gK) &= \int_{K/M} e^{(\mu-\rho)\beta( gk'K, k(gk'w_-)M)} \psi(k(gk'w_-)M)~\intd (k'M)\\
&= \int_{K/M} e^{(\mu-\rho)\beta( gK, k(gk')M)} \psi(k(gk')M)~\intd (k'M)\\
&= \int_{K/M} e^{(\mu-\rho)\beta( gK, kM)} \psi(kM)e^{-2\rho(H(g^{-1}k))}~\intd (kM)\\
&=  \int_{K/M} e^{(\mu-\rho)\beta( gK, kM)} \psi(kM)e^{2\rho\beta( gK, kM)}~\intd (kM)\\ 
&= (\mathcal P_\mu\psi)(gK)
 \end{align*}
(the third line corresponds to a change of variable $k'\mapsto k:=k(gk')$ in $\pl {\bf H}=K/M$). 
The equality $\mathcal P_\mu = \til{\pi}_{0*} \circ \mathcal Q_{\mu-\rho, +}$
follows analogously.
\end{proof}

\begin{thm}\label{thm:intertwining}
 Let $G$ be of real rank $1$ and $\Gamma\subset G$ a discrete, torsion-free, co-compact subgroup, let ${\bf H}=G/K$ and ${\bf M}=\Gamma\backslash {\bf H}$. If $\mu \in \a^*_\C\setminus\mathbf{Ex}$  is a regular spectral parameter,  then the pushforward $\pi_{0*}:\mathcal D'(S{\bf M})\to\mathcal D'({\bf M})$ restricts to isomorphisms
 \begin{equation}\label{isompi0}
  {\pi_0}_*:\tu{Res}^0_{X}(\mu-\rho)\to \tu{Eig}_{\Delta_{\bf M}}(\mu)
 \quad
\text{and}
\quad
  {\pi_0}_*:\tu{Res}^0_{X^*}(\mu-\rho)\to \tu{Eig}_{\Delta_{\bf M}}(\mu).
 \end{equation}
 \end{thm}
 
\begin{proof}
Note that  we have a bijection $(\til\pi_{\Gamma})^*: \tu{Eig}_{\Delta_{\bf M}}(\mu)\to{^\Gamma \mathcal E_\mu({\bf H})}$ (Remark~\ref{lem:lift_quantum}) 
as well as  bijections $(\pi_{\Gamma})^*:\tu{Res}_{X}^0(\lambda(H_0)) \to {^\Gamma \mc{R}_-(\lambda)}
$ and $(\pi_{\Gamma})^*:\tu{Res}_{X^*}^0(\lambda(H_0)) \to {^\Gamma \mc{R}_+(\lambda)}$ (Remark~\ref{lem:firs_band_Gamma_inv_vectors}). If we consider the projections on the  covers we have the following commuting diagram
\[ 
\begin{tikzcd} ^\Gamma\mathcal D'(S{\bf H}) \arrow{r}{\til{\pi}_{0*}}  & ^\Gamma\mathcal D'({\bf H}) \\
\mathcal D'(S{\bf M}) \arrow{u}{\pi_{\Gamma}^*}\arrow{r}{{\pi_0}_*} & \mathcal D'({\bf M})\arrow[swap]{u}{\til\pi_{\Gamma}^*}. 
\end{tikzcd} 
\]

We see that it is sufficient to prove that $\til{\pi}_{0*}: {^\Gamma \mc{R}_\pm(\mu-\rho)}\mapsto {^\Gamma \mathcal E_\mu({\bf H})}$
is an isomorphism. By Propositions~\ref{thm:poisson_trafo} and \ref{prop:classical_compact_intertwining}
we already know that $\mathcal P_\mu \circ(\mathcal Q_{\mu-\rho,\pm})^{-1}:
{^\Gamma \mc{R}_\pm(\mu-\rho)}\mapsto {^\Gamma \mathcal E_\mu({\bf H})}$ is an isomorphism and we conclude with Proposition~\ref{prop:poisson_projection}.
\end{proof}

%\begin{rem}
% The proof of Theorem~\ref{thm:intertwining} also gives an explicit form of the
% inverse map of the bijection $\til \pi_{0*}$ which we write as
% \[
%  \mathcal I_-(\mu) :\Abb{\tu{Eig}_{\Delta_{\bf M}}(\mu)}{\tu{Res}^0_{X}(\mu-\rho)}
%  {\phi}{\left[(\pi^*_{\Gamma})^{-1}\circ \mathcal Q_{\mu-\rho, -}\circ (\mathcal P_\mu)^{-1}\circ\pi^*_{\Gamma}\right](\phi)}
% \]
%and
%\[
%  \mathcal I_+(\mu) :\Abb{\tu{Eig}_{\Delta_{\bf M}}(\mu)}{\tu{Res}^0_{X^*}(\mu-\rho)}
%  {\phi}{\left[(\pi^*_{\Gamma})^{-1}\circ \mathcal Q_{\mu-\rho, +}\circ (\mathcal P_{\mu})^{-1}\circ\pi^*_{\Gamma}\right](\phi)}
% \]
%\end{rem}

We can also  give a precise description of first band Jordan blocks:  we write 
 \begin{align*}
\tu{Res}_{X}^0(\lambda, k):= \Big\{&u\in \mathcal D'(S{\bf M})\mid
(X+\lambda)^ku=0, U_-u=0, ~\forall U_- \in C^\infty(S{\bf M};E_u)\Big\}.
\end{align*}
\begin{thm}\label{thm:jordan}
If $\mu\in \a^*_\C\setminus \mathbf{Ex}$ is a regular spectral parameter and $\mu \neq0$,
 then there are no first band Jordan blocks with spectral parameter 
 $\lambda = -\rho+\mu$:
 \[
  \forall k\in \N^*~:~ \tu{Res}_{X}^0(\lambda, k)
  = \tu{Res}_{X}^0(\lambda, 1) .
 \]
 If $\mu=0$, then the first band Jordan block at  $\lambda = -\rho$ is 
 precisely of size two, i.e.
 \[\forall k\geq 2, \quad 
  \tu{Res}_{X}^0(\lambda, k)
  = \tu{Res}_{ X}^0(\lambda, 2)
\text{, and }
 \dim \tu{Res}_{X}^0(\lambda, 2) =2 \dim\tu{Res}_{X}^0(\lambda, 1).
 \]
\end{thm}

\begin{proof}
 Let $\mu\in \a^*_\C$, $\mu\notin \mathbf{Ex}\cup\{0\}$. In order to simplify  the notation, let us identify throughout this proof $\a^*_\C\cong \C$ by identifying
 $\nu_0$ with $1$. Assume that there is a non-trivial Jordan block for  $\lambda =-\rho+\mu$. Then there are nonzero distributions  $u_0, u_1\in \mathcal D'(S{\bf M})$ with  $(X +\lambda)u_0 =0$, $(X +\lambda)u_1=u_0$
 and $U_- u_{0}=U_-u_1=0$ for all smooth sections $U_-$ in $\til{E}_u$.  From Theorem~\ref{thm:intertwining} we know that $\phi_0:=\pi_{0*}( u_0)\in C^\infty({\bf M})$ is a nonzero element that fulfills $(\Delta_{\bf M} - \rho^2+\mu^2)\phi_0 =0$. We claim that by setting $\phi_1:= \pi_{0*}(u_1)$ we obtain
 \[
  (\Delta_{\bf M} - \rho^2+\mu^2)\phi_1 =2\mu\phi_0.
 \]
If this holds, then pairing this equation with $\phi_0$, for $\mu\not=0$ we get $2\mu||\phi_0||_{L^2}^2=0$ so that $\phi_0=0$, which is a contradiction.

Now let us prove the claim: let $\til u_0, \til u_1\in \mathcal D'(S{\bf M})$ and $\til \phi_0, \til \phi_1 \in C^\infty ({\bf H})$ be the lifts of $\phi_i$ to the universal cover. By Proposition~\ref{prop:ruell_jordan_bv} and its proof we can write $\til u_1 = \Phi_-^{\mu-\rho}(B_-^*T_1 - \log \Phi_- B_-^*T_0)$ with $T_0, T_1\in \mathcal D'(\pl {\bf H})$. Furthermore, by Proposition~\ref{prop:poisson_projection} we can write $\partial_\mu P_\mu(T_0) = \til \pi_{0*}(\Phi_-^{\mu-\rho}\log\Phi_- B_-^*T_0)$. Thus we get $\til \phi_1 = \mathcal P_\mu (T_1) - \partial_\mu \mathcal P_\mu(T_0)$ and
\begin{equation}
 \label{eq:laplace_jordan_1}
 (\Delta_{{\bf H}} - \rho^2+\mu^2)\til{\phi}_1 =
 -(\Delta_{{\bf H}} - \rho^2+\mu^2)\partial_\mu\mathcal P_\mu(T_0).
\end{equation}
Now it is easily checked that 
\begin{equation}
 \label{eq:laplace_jordan_2}
2\mu \mathcal P_\mu(T_0) +(\Delta_{{\bf H}} -\rho^2+\mu^2)\partial_\mu\mathcal P_\mu(T_0) =0,
\end{equation}
by taking the derivate of the equation $(\Delta_{{\bf H}} -\rho^2+\mu^2)\mathcal P_\mu(T_0)=0$ with respect to $\mu$. Thus, \eqref{eq:laplace_jordan_1} and \eqref{eq:laplace_jordan_2} imply the above claim.

Let us finally consider the case $\mu = 0$ and assume there is a Jordan block of size larger than two. Then there are nonzero distributions $u_0,u_1,u_2 \in \mathcal D'(S{\bf M})$ such that
\[
 (X -\rho)u_2= u_1,~~(X -\rho)u_1= u_0,~~(X -\rho)u_0= 0
\]
and $U_- u_{i}=0$ for all smooth sections $U_-$ in $\til{E}_u$. We set for $i=1,2,3$, $\phi_i = \pi_{0*}(u_i)\in \mathcal D'({\bf M})$ and claim that
\begin{equation}
 \label{eq:jordan_claim_mu_0}
 (\Delta_{\bf M} -\rho^2)\phi_2= -\phi_0.
\end{equation}
As above, pairing this equation with $\phi_0$ we get a direct contradiction. In order to prove this claim, let $\til u_0, \til u_1,\til u_2 \in \mathcal D'(S{\bf H})$ and $\til \phi_0, \til \phi_1, \til \phi_2 \in \mathcal D' ({\bf H})$ be the lifts to the covers. By Proposition~\ref{prop:ruell_jordan_bv} and its proof we can write
\[
\til u_2 = \Phi_-^{-\rho}\left(B_-^*T_2 - (\log \Phi_-) B_-^*T_1+ \frac{1}{2} (\log \Phi_-)^2B_-^*T_0\right)
\]
with $T_0, T_1, T_2\in \mathcal D'(\pl{\bf H})$.

Using again the representation of the Poisson transformation from Proposition~\ref{prop:poisson_projection} we get
\begin{equation}
 \label{eq:jordan_mu_0_a}
\til \phi_2 = \mathcal P_\mu (T_2) - (\partial_\mu \mathcal P_\mu)_{|\mu=0}(T_1) +\frac{1}{2}
(\partial_\mu^2 \mathcal P_\mu)_{|\mu=0}(T_0).
\end{equation}
Now taking the second derivative of $(\Delta_{\bf H} - \rho^2+\mu^2)\mathcal P_\mu(T_0)=0$ at $\mu=0$ we get
\begin{equation}\label{eq:jordan_mu_0_b}
2\mathcal P_0(T_0)  + (\Delta_{{\bf H}} - \rho^2) (\partial_\mu^2\mathcal P_\mu)_{|\mu=0}(T_0)=0. 
\end{equation}
Finally \eqref{eq:jordan_mu_0_a}, \eqref{eq:laplace_jordan_2} and \eqref{eq:jordan_mu_0_b}
together imply \eqref{eq:jordan_claim_mu_0}. This shows that the first band Jordan blocks of $X$ at $\lambda=-\rho$ are maximally of size two.

Let us finally show that at $\mu=0$ the Jordan blocks are at least of size two. To this end we assume that $\tu{Res}_{X}^0(-\rho)$ is nonzero and $u_0\in \tu{Res}_{X}^0(-\rho)\setminus\{0\}$. By Proposition~\ref{prop:classical_compact_intertwining} we know that for $\til u_0 \in \mc{R}_-(-\rho)$ there exists a classical boundary value $T_0\in \mathcal D'(\pl {\bf H})$ such that $\til u_0 = \mathcal Q_{-\rho, -} (T_0)$. From the same proposition we obtain $\pi_0^\tu {cpt}(\gamma) T_0 =T_0$ for all $\gamma\in \Gamma$. For
$\mu\notin \a^*_\C\setminus \mathbf{Ex}$ we define the scattering operator for ${\bf H}$
\begin{equation}\label{eq:scattering_operator}
S_\mu:\Abb{\mathcal D'(\pl{\bf H})}{\mathcal D'(\pl{\bf H})}{T}{\mathcal P_{\mu}^{-1}\circ \mathcal P_{-\mu}(T)}.
\end{equation}
Note that by Proposition~\ref{thm:poisson_trafo} this operator is well defined and we have
\begin{equation}\label{eq:scattering_intertwining}
S_\mu\pi_\mu^\tu{cpt}(g) = \pi_{-\mu}^\tu{cpt}(g) S_\mu \tu{ for all }g\in G.
\end{equation}
Obviously, we have $S_0 = \tu{Id}$. Taking the derivative of \eqref{eq:scattering_intertwining}
with respect to $\mu$ at $\mu=0$, we obtain
\begin{equation}
 \label{eq:scattering_intertwining_2}
  \pi_0^\tu{cpt}(g) (\partial_\mu S_\mu)_{|\mu=0} - (\partial_\mu S_\mu)_{|\mu=0} \pi_0^\tu{cpt}(g) = -2(\partial_\mu \pi_{-\mu}^\tu{cpt}(g))_{|\mu=0} \tu{ for all }g\in G.
\end{equation}
If we now set $T_1 :=-\tfrac{1}{2}(\partial_\mu S_\mu)_{|\mu=0}T_0 \in \mathcal D'(\pl{\bf H})$, then
\eqref{eq:scattering_intertwining_2}, together with the transformation properties of $T_0$, yields
\[
 \pi_0^\tu{cpt}(\gamma)T_1-T_1 = (\partial_\mu \pi^\tu{cpt}_{-\mu}(\gamma))_{|\mu=0}T_0) \tu{ for all } \gamma\in \Gamma.
\]
With \eqref{eq:compact_picture_N_gamma} this implies
\[
 \gamma^* T_1 = N_\gamma^\rho(T_1+\log N_\gamma T_0) \tu{ for all } \gamma\in \Gamma.
\]
Note that Proposition~\ref{prop:ruell_jordan_bv} implies that the existence of such a pair of distributions $T_0, T_1\in \mathcal D'(\pl{\bf H})$ is equivalent to the existence of $u_1\in \mathcal D'(S{\bf M})$ with $(X -\rho)u_1=u_0$. We have thus constructed for any $u_0\in \tu{Res}^0_{X}(-\rho)$ a Jordan block of size two. This finishes the proof of Theorem~\ref{thm:jordan}.
\end{proof}

For the special case of a real hyperbolic space a description of Jordan blocks is given in \cite{DFG15}. The proof there however relies on the pairing formula and the self adjointness of  the Laplacian. Here we have given a different proof, more in the spirit of
\cite{GHW18}, that also allows the precise description of the spectral value at $\lambda=-\rho$ which was untractable with the methods in \cite{DFG15}. For $G={\rm SL}_2(\rr)$, it is also shown 
in \cite{Co05,FF03} that $-\rho$ has a Jordan block.

Let us finally give a rough description of the first band resonant states at the exceptional points.
At these points the Poisson transformation is not injective anymore, but one has a nontrivial, closed $G$-invariant subspace $\ker P_\mu \subset \mathcal D'(\pl {\bf H})\cong H_\mu^{-\infty}$ as well as a $G$-invariant
subspace ${\rm Im}\, P_\mu \subset \mathcal E_\mu^*({\bf H})$. In particular, the existence of the kernel implies that at these exceptional points there could be more Ruelle resonant states in the first band than expected from the Laplace spectrum.

\begin{prop}
 Let $\mu\in \mathbf{Ex}\subset \a^*_\C$ be an exceptional spectral parameter. Then
 \begin{equation}
  \label{eq:exceptional_1}
  \dim \tu{Res}^0_{X}(\mu-\rho) = \dim \tu{Res}^0_{X^*}(\mu-\rho)
  = \dim (^\Gamma\ker P_\mu) + \dim (^\Gamma{\rm Im}\,P_\mu).
 \end{equation}
 In particular, if $\|\mu\| > \|\rho\|$,  then we have
 \begin{equation}
  \label{eq:exceptional_2}
   \dim \tu{Res}^0_{X}(\mu-\rho) = \dim \tu{Res}^0_{X^*}(\mu-\rho)
  = \dim (^\Gamma\ker P_\mu).
 \end{equation}
If $\mu=-\rho$, then we get $ \dim \tu{Res}^0_{X}(\mu-\rho) = \dim \tu{Res}^0_{X^*}(\mu-\rho)  = \dim (^\Gamma\ker P_\mu)+1$.
\end{prop}

\begin{proof}
The statement follows directly from Proposition~\ref{prop:ruelle_gamma_inv_distr}. If we write the finite dimensional vector space $^\Gamma(H^{-\infty}_{\mu}) = {^\Gamma(\ker P_\mu)}\oplus C$ with an arbitrary linear complement $C\subset{} ^\Gamma(H^{-\infty}_{\mu})$, then $P_\mu: C\to{^\Gamma {\rm Im}\,P_\mu}$ is a bijection and we get  (\ref{eq:exceptional_1}).
Furthermore, if $\|\mu\| > \|\rho\|$,
then by positivity of $\Delta_{\bf M}$ we know that
$^\Gamma \mathcal E^*_\mu({\bf H}) \cong \tu{Eig}_{\Delta_{\bf M}}(\mu) = 0$, thus
$\dim (^\Gamma {\rm Im}\,P_\mu)=0$ and we obtain (\ref{eq:exceptional_2}). If $\mu=-\rho$, then ${\rm Im}\,P_\mu=\cc$.
\end{proof}

Note that for real and complex hyperbolic spaces,  $\|\mu\|\geq \|\rho\|$ if $\mu\in \mathbf{Ex}$.

\section{A new description of Patterson-Sullivan distributions}\label{PatSul}

We briefly recall the construction of Patterson-Sullivan distributions from \cite{AZ07,HS09,HHS12}.
For $\mu,\mu'\in \a^*_\C$ we introduce a weighted Radon transformation $\mathcal R_{\mu,\mu'}$ by  
\[
 (\mathcal R_{\mu,\mu'} f)(g):=\int_A e^{(\mu+\rho)H(ga) + (\mu'+\rho)H(gaw)}f(ga) \intd a
\]
for each $f\in C_c^\infty(G/M)$.
By definition this Radon transformation is right $MA$-invariant and we have
\cite[Lemma 4.4]{HHS12}
\[
 \mathcal R_{\mu,\mu'}:C_c^\infty(G/M)\to C_c^\infty(G/MA).
\]
The Patterson-Sullivan distributions are then defined as follows:

\begin{Def}%[cf. \cite[Def.~4.8]{HHS12}, Definition 6.5 \cite{HS09}]
\label{def:PS}
 Let $\mu,\mu'\in \a_{\cc}^*$ and $\varphi_\mu\in\mathcal E^*_\mu(G/K)$
 $\varphi_{\mu'}\in\mathcal E^*_{\mu'}(G/K)$. Then the associated \emph{Patterson-Sullivan distribution}
 $\operatorname{PS}_{\varphi_\mu,\varphi_{\mu'}} \in \mc{D}'(G/M)$ is the generalized density
 defined by its evaluation at $f\in C_c^\infty(G/M)$:
 \[
  \operatorname{PS}_{\varphi_\mu,\varphi_{\mu'}}(f):= \int\limits_{G/MA}(\mathcal R_{\mu,\mu'} f)(gMA)
[P_\mu^{-1}(\varphi_\mu)](\intd b)\otimes [P_{\mu'}^{-1}(\overline{\varphi_{\mu'}})](\intd b').
 \]
Here $[P_\mu^{-1}(\varphi_\mu)](\intd b)$ and
$[P_{\mu'}^{-1}(\overline{\varphi_{\mu'}})](\intd b')$ are the generalized densities
on $\pl{\bf H}=K/M=G/P$ obtained by the boundary distributions and the invariant measure. Their tensor
product is a generalized density on $\pl{\bf H}\times \pl{\bf H}$ and can be
restricted to the open subset $(\pl{\bf H})^2_{\Delta}:= (\pl{\bf H}\times \pl{\bf H})\setminus \Delta(\pl{\bf H})$, where $\Delta(\pl{\bf H})$ is the diagonal in $\pl{\bf H}\times \pl{\bf H}$. Note that $G$ acts transitively on $(\pl{\bf H})_\Delta^2$ with respect to the diagonal action and that 
$(\pl{\bf H})_\Delta^2\cong G/MA$ as a $G$-homogeneous space.
\end{Def}
 
Let us denote by $\mc{I}_+(\mu)$ and $\mc{I}_-(\mu)$ the respective inverse maps of 
the isomorphisms \eqref{isompi0}, defined on $\Eig_{\Delta_{\bf M}}(\mu)$ by 
\[\mc{I}_+(\mu):=(\pi_\Gamma^*)^{-1}\mathcal Q_{\mu-\rho,+}\circ(P_{\mu})^{-1}\til\pi_\Gamma^*,\quad  
\mc{I}_-(\mu):=(\pi_\Gamma^*)^{-1}\mathcal Q_{\mu-\rho,-}\circ(P_{\mu})^{-1}\til\pi_\Gamma^*.\]
\begin{thm}\label{thm:patterson_sullivan}
 Let $\mu,\mu'\in \a^*_\C\setminus\mathbf{Ex}$ and $\varphi_\mu \in \Eig_{\Delta_{\bf M}}(\mu)$ and
 $\varphi_{\mu'} \in \Eig_{\Delta_{\bf M}}(\mu')$. Then the Patterson-Sullivan
 distribution $ \operatorname{PS}_{\varphi_\mu,\varphi_{\mu'}}$ descends to 
 $S{\bf M}=\Gamma\backslash G/M$ and is given by
 \begin{equation}
  \label{eq:PS_identification}
  \operatorname{PS}_{\varphi_\mu,\varphi_{\mu'}} = \mathcal I_+(\mu)( \varphi_{\mu})\cdot \mathcal I_-(\mu')(\overline{\varphi_{\mu'}}).
 \end{equation}
 Here the product of the distributions $\mathcal I_+(\mu)(\varphi_\mu)$ and $\mathcal I_-(\mu')(\overline{\varphi_{\mu'}})$ is well defined by the wavefront condition. The descended Patterson-Sullivan distribution will be  denoted by $\operatorname{PS}^\Gamma_{\varphi_\mu,\varphi_{\mu'}}$.
\end{thm}
\begin{proof}
We will prove the statement for the corresponding $\Gamma$-invariant distributions
on $G/M$. Let $\til \varphi_\mu\in \mathcal E_\mu(G/K)$ and
$\til\varphi_{\mu'}\in \mathcal E_{\mu'}(G/K)$ be the lifted Laplace eigenfunctions.
Then the lift of the left hand side of \eqref{eq:PS_identification} simply becomes
$\operatorname{PS}_{\til \varphi_\mu.\til \varphi_{\mu'}}$, while the right hand
side becomes
\[
 \left[[\mathcal Q_{\mu'-\rho,-}\circ(\mathcal P_{\mu'})^{-1}](\overline {\til\varphi_\mu})\cdot
  [\mathcal Q_{\mu-\rho,+}\circ(\mathcal P_{\mu})^{-1}]
 (\til\varphi_{\mu})\right](\intd (gM))
 \in\mc{D}'(G/M).
\]
Note that with the diffeomorphism
\[
 \Psi:\Abb{G/M}{A\times \mathcal (\pl{\bf H})_\Delta^2}{gM}{(a(g), B_+(gM), B_-(gM))}
\]
we can write the $G$ invariant measure $\intd (gM)$ as
\[
 \intd gM = e^{2\rho(H(g) + H(gw))} (\Psi^{-1})_* (\intd a\intd b\intd b').
 \]
In fact, in view of $G/MA\cong (\pl{\bf H})^2_\Delta$ and the description of the $K$-invariant measures on $G/P\cong K/M$ and $G/\overline P\cong K/M$ this follows from 
\[\int_{G/M} f(gM)~\intd (gM) =
\int_{G/MA}\int_A f(gaM)~\intd a\intd (gMA)\]
and $\Psi(gaM)=(a(g)a,B_+(gM), B_-(gM))$.
Now inserting the expressions for $\mathcal Q_{\lambda,\pm}$ from \eqref{eq:Q_lambda_def}
and replacing the measure, we obtain for the right-hand-side
\[
 e^{(\mu+\rho)H(g) + (\mu'+\rho)H(gw)}(\Psi^{-1})_*\left(da  \otimes
 [\mathcal P_{\mu}^{-1}(\til\varphi_{\mu})](db)
 \otimes [\mathcal P_{\mu'}^{-1}(\overline{\til\varphi_{\mu'}})](db')
 \right).
\]
This is exactly the Patterson-Sullivan distribution as defined in
Definition~\ref{def:PS}.
\end{proof}

\section{A pairing formula for Ruelle resonant states}

Let us normalize the Haar measure\footnote{This normalization is slightly different from normalizations common in the literature of symmetric spaces, where one has $c(\rho)=1$. We have chosen this normalization such that they give simple formulas in our geometric context.} on $\bN$ by
$\int_\bN e^{-2\rho H(\bn)} \intd \bn =\tu{vol}(\mathbb S^{n-1})$. Now we can define the Harish-Chandra c-function as the holomorphic function given by the convergent integral for $\tu{Re}(z)>0$ 
\begin{equation}\label{HCc}
 c(z):=\int_\bN e^{-(\rho+z)H(\bn)} \intd \bn.
\end{equation}
It has a meromorphic continuation to $z\in \C$ (see e.g. \cite[IV.6]{Hel84}) which is given by
\[
 c(\lambda\|\alpha_0\|) = c_0\frac{2^{-\lambda}\Gamma(\lambda)}{\Gamma(\frac{1}{2}(\frac{1}{2}m_{\alpha_0} + 1 +\lambda))\Gamma(\frac{1}{2}(\frac{1}{2}m_{\alpha_0}
 + m_{2\alpha_0} + \lambda))},
\]
where 
\[
 c_0 = \frac{\pi^{n/2}2^{1+\frac{1}{2}m_{\alpha_0} + m_{2\alpha_0}}\Gamma(\frac{1}{2}(m_{\alpha_0} + m_{2\alpha_0}+1))}{\Gamma(n/2)}.
\]
One easily checks that the zeros and poles of the c-function are contained in the real line. 
 
\begin{thm}\label{thm:pairing_formula}
Let $\lambda\in \C\setminus\{-\rho - \N_0\alpha_0\}$ 
be a Ruelle resonance in the first band and let $v\in \Res^0_X(\lambda)$ and $v^ *\in
\Res^0_{X^*}(\lambda)$ be some associated resonant/co-resonant states. Then we have 
\[
 \int_{\mathbf M} {\pi_0}_*(v) {\pi_0}_*(v^*) \intd x = 
 c(\rho+\lambda)  v\cdot v^* [1_{S\mathbf M}].
\]
The product $v\cdot v^*$ is well defined by the 
wavefront set properties of $v$ and $v^*$ and can be paired with the constant one
function $1_{S\mathbf M}$. 
\end{thm}

As a direct consequence of Theorem~\ref{thm:pairing_formula} we obtain that for two
quantum eigenstates $\varphi,\varphi'\in \tu{Eig}_{\Delta_{\mathbf M}}(\mu)$ the normalization of the corresponding Patterson-Sullivan distribution is given by
\[
 \operatorname{PS}_{\varphi,\varphi'}[1_{S\mathbf M}] = \frac{1}{c(\mu)} \langle\varphi,\varphi'\rangle_{L^2(\mathbf M)}.
\]
Furthermore, the pairing formula allows to relate the invariant Ruelle distributions 
(defined in \eqref{eq:inv_ruelle_dist}) to Patterson-Sullivan distributions:

\begin{cor}
 Let $r>0$ such that $-\rho+ir$ is a Ruelle resonance of multiplicity $m$ in the first band. Then $\rho^2+r^2$ is an eigenvalue of $\Delta_{\bf M}$ with multiplicity $m$ and, for
 an $L^2$-orthonormal basis $\varphi_1,\ldots,\varphi_m$ of 
 $\tu{Eig}_{\Delta_{\mathbf M}}(ir)$, we have
 \[
  \mathcal T_{-\rho+ir} = \frac{c(ir)}{m} \sum_{l=1}^m \operatorname{PS}_{\varphi_l, \varphi_l}.
 \]
\end{cor}

\begin{proof}
 Via the quantum-classical correspondence (Theorem~\ref{thm:intertwining}) we define a basis 
 of Ruelle resonant states $u_l:= \mathcal I_-(ir)(\varphi_{l}) \in \tu{Res}_{X}(-\rho+ir)$ 
 as well as a basis of co-resonant states $u_l^*:= c(ir) \mathcal I_+(ir)(\overline{\varphi_{l}})\in
\tu{Res}_{X^*}(-\rho+ir)$.
Now the pairing formula (Theorem~\ref{thm:pairing_formula}) together with the chosen normalization 
of $u_i^*$ implies the bilinear pairing formula 
\[\cjg u_i^*,u_j\cjd := \int_{S\mathbf M} u_i\cdot u_i^* d\mu_{\tu{L}} = \delta_{ij},\]
which means that the basis $u_i^*$ is dual to the chosen basis $u_i$. 
By Theorem~\ref{thm:jordan} we know that there are no Jordan blocks
at the spectral parameter $-\rho+ir$ and thus the spectral projector can be written
as $\Pi_{-\rho+ir} = \sum_{l=1}^{m} u_l\otimes u_l^*$. Now for $f\in C^{\infty}(\mathcal M)$
we obtain
\[
 \mathcal T_{\rho+ir}[f] = \frac{1}{m} \sum_{l=1}^{m} \cjg u_l^*,f u_l\cjd
 = \frac{c(ir)}{m}\sum_{l=1}^{m} \operatorname{PS}_{\varphi_{l}, \varphi_{l}}[f]
\]
which completes the proof.
\end{proof}

So far, pairing formulas like Theorem~\ref{thm:pairing_formula} have been 
shown for compact hyperbolic surfaces \cite[Theorem 1.2]{AZ07} and compact real hyperbolic manifolds \cite[Lemma 5.10]{DFG15}. We follow the strategy of proof of \cite{DFG15}: first we consider
$\int_{\mathbf M} {\pi_0}_*(v){\pi_0}_*(v^*)$ with respect to the measure which is obtained by restricting the measure $\intd x$ on $G/K$ to a $\Gamma$-fundamental domain. Then we construct a coordinate transformationation on an open dense subset that formally makes  the 
Harish-Chandra c-function appear as an integral over $N$ (see Lemma~\ref{lem:pairing_coordinate}).
However, as we integrate distributions, we have to cut out an $\epsilon$-neighborhood of the points where the coordinate transformation is not defined and consider the limit $\epsilon\to 0$ (see Lemma~\ref{lem:pairing_asympt_expansion}).
As in \cite{DFG15} this turns out to be a subtle limit that requires a suitable regularization 
of divergent integrals. Compared to the constant curvature case there are two major challenges: 
first, one has to replace the explicit computations in the hyperboloid model in the construction of the coordinate transformations. Second, one has to deal with the fact that the
defining integral of the c-function becomes an integral over a non-commutative group $N$,  
adding some anisotropy to the regularization process. 

\subsection{A suitable coordinate transformation}
Let us define the double unit sphere bundle $\mathbb S^ {n-1}\times \mathbb S^ {n-1}\to S^2\mathbf{M}\to \mathbf M$  as the pullback of $S\mathbf M\times S\mathbf M$ under the diagonal embedding $\mathbf M \hookrightarrow \mathbf M\times\mathbf M$ (or equivalently as $S^2\mathbf{M}=\{((x,\eta_+),(x',\eta_-))\in S{\bf M}\x S{\bf M}\mid x=x'\}$) and equip this bundle with the measure $\intd x\otimes\intd{\mu}_{\mathbb S^{n-1}}\otimes\intd{\mu}_{\mathbb S^{n-1}}$ for which we will use the slightly shorter notation
$\intd x\intd \eta_+\intd\eta_-$. Define
\[
 I:=\int_{\mathbf M} {\pi_0}_*(v) {\pi_0}_*(v^*) \intd x = \int_{S^2\mathbf M} v(x,\eta_+)v^*(x,\eta_-) \intd x\intd \eta_+\intd\eta_-.
\]
The  integral on the right makes sense by the wavefront set properties of $v,v^*$: indeed,
$v\otimes v^*\in \mc{D}'(S{\bf M}\x S{\bf M})$ has wavefront set contained in 
$E_u^* \x E_s^*$ and its restriction to the submanifold
$S^2{\bf M}$ makes sense since $N^*(S^2{\bf M})\cap (E_u^* \x E_s^*)=\emptyset$ because of the transversality $E_u^*\cap E_s^*=\{0\}$.
We furthermore define the open dense subset $S^2_\Delta\mathbf M := \{(x,\eta_-,\eta_+)\in S^2\mathbf M \mid \eta_-+\eta_+ \neq 0\}$. Recall that the unstable bundle $E_u\to\mathbf M$ is given by an associated vector bundle $E_u = \Gamma\backslash G\times_M\n_-$ and, using the exponential map, can be identified with $\Gamma\backslash G\times_M\bN$. We will denote points in $E_u$ by equivalence classes $[g,\overline n]$, where $[gm,\bn]= [g,m\bn m^{-1}]$ for all $m\in M$. Thus, any $M$-conjugation invariant function $\chi\in C_c^\infty(\bN)$ defines a function in $C_c^\infty(E_u,\C)$, which we also denote by $\chi$. We have:

\begin{lem}\label{lem:pairing_coordinate}
 There is a diffeomorphism $\mathcal A :S^{2}_\Delta \mathbf M \to E_u$ 
 such that for any $M$-conjugation invariant function 
 $\chi\in C_c^\infty(\bN)$, any Ruelle resonance 
 $\lambda\in \C$ in the first band and associated resonant/co-resonant states 
 $v\in \Res^0_X(\lambda)$, $v^ *\in \Res^0_{X^*}(\lambda)$, we have
 \begin{equation}\label{eq:pairing_coordinate}
   \begin{split}
     \int_{S^2_\Delta \mathbf M} v(x,\eta_-)v^*(x,\eta_+)\chi\circ&\mathcal A  (x,\eta_-,\eta_+)\intd x\intd \eta_- \intd \eta_+ =\\
     &\int_{S\mathbf M} v v^*d\mu_{\tu{L}}\int_{\bN} e^{-(2\rho+\lambda)H(\bn)}\chi(\bn) \intd \bn 
   \end{split}
 \end{equation}
\end{lem}

\begin{proof}
\cpic{0.5\textwidth}{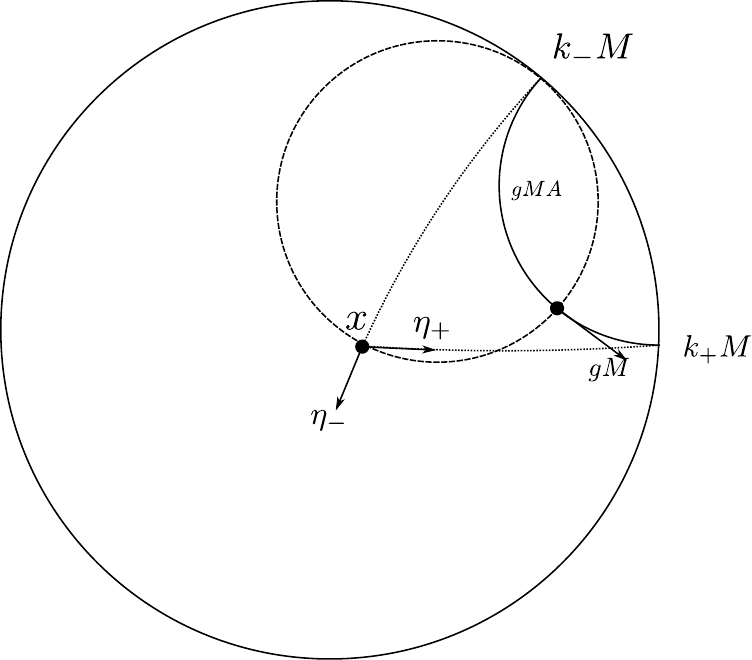}{The figure visualizes the coordinate transformation 
for the case $G=SL(2,\R)$: A base point $x\in G/K$ can be thought as a point in the Poincar\'e disk
and $\eta_\pm$ are tangent vectors of this point. The points $k_\pm M\in K/M$ are the starting 
and end points through the geodesics, tangent to the vectors $\eta_\pm$. Except in the case 
$\eta_++\eta_-=0$ those are two disjoint points on the boundary $K/M$ and they define a 
unique geodesic which can be interpreted by $gMA$. On this geodesic, the point $gM$ 
is by construction chosen such that it is the projection of $x$ along an unstable manifold of the 
geodesic $gMA$.}{pairing_visualization}{}
(For a geometric interpretation of the constructions below, see 
Figure~\ref{fig:pairing_visualization}). We first give an explicit construction
of the diffeomorphism by constructing a left $G$-invariant diffeomorphism 
 $\til{\mathcal A }: S^2_\Delta\mathbf{H}^n \to \til E_u \cong G\times_M\bN$.  
 In a first step consider the $G$-equivariant diffeomorphism
 \[
  \mathcal E:\Abb{S^2\hspc}{\hspc\times\bhspc\times \bhspc}{(x,\eta_-,\eta_+)}{(x,B_-(x,\eta_-),B_+(x,\eta_+))}
 \]
 via the initial and end point maps $B_\pm: S\hspc \to \bhspc$ from \eqref{eq:initial_end_point_map} and note that it restricts to a diffeomorphism $\mathcal E: S^2_\Delta\hspc \to\hspc \times (\bhspc)^2_\Delta$.
 Furthermore, recall 
 (see e.g. \cite[Section 2]{HS09} and \cite[Prop.~2.7]{HHS12}) that 
 there is a $G$-equivariant diffeomorphism
 \begin{equation}\label{eq:def_G}
  \mathcal G:\Abb{G/MA}{(\bhspc)^2_\Delta}{gMA}{(gwP, gP)}
 \end{equation}
after identifying $\bhspc\cong G/P$, where $w=w_-$ denotes a representative
of the nontrivial Weyl group element. Note that
the map
\begin{equation}\label{eq:def-psi}
 \psi:\Abb{G\times_M \bN}{G/K  \times G/MA}{{[g,\bn]}}{(g\bn K, gMA)}
\end{equation}
is well defined and left $G$-equivariant. 
Furthermore, we can construct an inverse using the $A\bN K$-decomposition 
$g= a_{A\bN K}(g)\overline{n}_{A\bN K}(g) k_{A\bN K}(g)$
\[
 \psi^ {-1}: (hK, gMA)\mapsto [ga_{A\bN K}(g^ {-1}h), \overline{n}_{A\bN K}(g^ {-1}h)].
\]
Using these three diffeomorphisms we define $\mathcal A:= \psi^{-1}\circ\mathcal G^{-1}\circ\mathcal E$.

It remains to prove \eqref{eq:pairing_coordinate}. We clearly have that the left hand side equals
\[
 \int_{E_u}  \mc{A}_*(vv^*) \chi(\bn)\mathcal A_*(\intd x\intd \eta_-\intd \eta_+).
\]
Thus, we have to compute the pushforward of the distribution $\mc{A}_*(v v^*)$ and the 
measure $\mc{A}_*(\intd x\intd \eta_-\intd \eta_+)$.

Let us first consider the pushforward $\mathcal A_*v^*$:
as $v^*$ is a first band co-resonant state, we know by
Proposition~\ref{prop:classical_compact_intertwining} that 
$v^*=\Phi_+^\lambda  B_+^*w$ for 
some distribution $w\in \mathcal D'(\bhspc)$. If we write 
$\mathcal A(x,\eta_-,\eta_+)=[g,\bn]$, then
by the construction of $\mathcal A$ we have $B_+(x,\eta_+)= B_+(gM)=k_{KAN}(g)M$ 
and consequently
\[
(\mathcal A_*v^*)([g,\bn]) = v^*(x,\eta_+) = 
e^{\lambda(\beta( x, k_{KAN}(g)M)-\beta( gK, k_{KAN}(g)M))} v^*(gM).
\]
Since $x=g\bn K$, we can use \eqref{eq:def_horocycle} and \eqref{eq:horo_group_action1}  to  
simplify this expression to 
\[
(\mathcal A_*v^*)([g,\bn]) = e^{\lambda\beta( \bn k, g^{-1}k_{KAN}(g)M)} v^*(gM)= e^{-\lambda H(\bn^{-1})} v^*(gM).
\]
In a similar manner we obtain the equality $(\mathcal A_*v)([g,\bn]) = v(gM)$. The latter equality can also be understood
geometrically as $(x,\eta_-)$ and $gM$ lie, by construction, on the same unstable manifold 
and $v$ is constant along the unstable leaves, because it is a first band co-resonant state.

Next, let us consider the transformation of measures. In analogy to the final 
step in the proof of Proposition~\ref{prop:poisson_projection} we use 
\cite[Lemma I.5.19]{Hel84} to establish the formula
\[
 \mathcal E_*(\intd x\intd \eta_-\intd \eta_+) = 
 e^{2\rho(\beta( x,k_-M) + \beta( x,k_+M))}\intd x\intd (k_-M)\intd (k_+M).
\]
Moreover, by the Propositions~\ref{prop:Mass1} and \ref{prop:Mass2} proved in the appendix, we have 
\[
 \mathcal G_*(e^{-2\rho(H(g) + H(gw))}\intd (gMA)) = c_{\mathcal G}\intd (k_-M) \intd (k_+M)
\]
and 
\[
 c_\psi \psi_*(\intd (gM)\intd \bn) = \intd (gK)\intd (gMA).
\]
Putting these three transformations together and simplifying the exponents using \eqref{eq:def_horocycle} and \eqref{eq:horo_group_action1} we obtain 
a constant $c_{\mathcal A}>0$ such that 
\begin{equation}
 \label{eq:push_forward_measure}
\mathcal A_*(\intd x\intd \eta_-\intd \eta_+)=  c_{\mathcal A} e^{-2\rho H(\bn^{-1})} \intd (gM)\intd \bn.
\end{equation}
Because $\intd\bn$ is invariant under inversion, this establishes
\eqref{eq:pairing_coordinate} up to a multiplicative constant. However, 
this constant is equal to $1$ by the choice of the normalizations of the measures. 
This can be seen by passing to a co-compact quotient and integrating a constant function on
both sides of the variable transformation. This completes the proof of 
Lemma~\ref{lem:pairing_coordinate}.
\end{proof}

\subsection{Renormalization}
The formula of Lemma~\ref{lem:pairing_coordinate} would directly
imply the desired pairing formula if we could set $\chi=1$. However, as $v,v^*$ are distributions, this is not allowed and in fact we see for the 
important case of taking a Ruelle resonance with
$\Re(\lambda)=-\rho$ that the integral over $\bN$ on the right hand side of 
\eqref{eq:pairing_coordinate} would not converge anymore. We thus have to 
perform a careful regularization of the appearing quantities.

As a first step we introduce a suitable cutoff function: We take an arbitrary 
$M$-conjugation invariant function $\chi\in C_c^\infty(\bN)$ which is equal to 
$1$ in a neighborhood of the identity. Then we define for any $\varepsilon>0$
\[
 \chi_\varepsilon(\bn) := \chi(e^{-\ln(\varepsilon)\cH}\bn e^{\ln(\varepsilon)\cH}),
\]
where $\check H_0:= H_0/\|\alpha_0\|$. Since the conjugation by $\exp({-\ln(\varepsilon)\cH})$ strictly contracts 
$\bN$ to the identity, the function $\chi_\varepsilon$ converges to $1$ on compact sets of $\bN$ as $\varepsilon \to 0$.
As $M$ centralizes $A$, $\chi_\varepsilon$
is still $M$-conjugation invariant, so that it defines a function in
$C_c^\infty(\Gamma\backslash G\times_M \bN, \C)$. After a pullback with 
$\mathcal A$ this function can be identified with a function in 
$C_c^ \infty(S^2_\Delta\mathbf M)\subset C^\infty(S^2\mathbf M)$. By abuse of 
notation we will denote all these instances of the function by $\chi_\varepsilon$ 
and it will be clear from the context where the function lives. Having 
defined $\chi_\varepsilon$ we can write
\[
 I=\underbrace{\int_{S^2_\Delta\mathbf M} v(x,\eta_-)v^*(x,\eta_+) \chi_\varepsilon \intd x\intd 
 \eta_-\intd \eta_+}_{=:I_c(\varepsilon)} 
 +\underbrace{\int_{S^2\mathbf M} v(x,\eta_-)v^*(x,\eta_+) (1- \chi_\varepsilon) \intd x\intd 
 \eta_-\intd \eta_+}_{=:I_0(\varepsilon)}.
\]
Note that by this decomposition the term $I_c(\varepsilon)$ can now be treated with the
 transformation from Lemma~\ref{lem:pairing_coordinate}. The idea of
taking formally $\chi=1$ would correspond to considering the limit $\varepsilon\to 0$.
We will see that in general this limit is defined neither for $I_c(\varepsilon)$ nor 
for $I_0(\varepsilon)$. However, we can prove the following important asymptotic expansions.

\begin{lem}\label{lem:pairing_asympt_expansion}
 Given a  Ruelle resonance $\lambda\in \C\setminus\{-\rho - \N_0\alpha_0\}$ in the first band  and corresponding Ruelle resonant/coresonant states
 $v\in\Res_X^0(\lambda), v^*\in\Res_{X*}^0(\lambda)$, there are nonzero 
 exponents  $\beta_\ell\in\C$, $\ell=1,\ldots, \ell_0$ such that for some 
 coefficients $\alpha_\ell \in\C$
 \begin{equation}\label{eq:pairing_asympt_expansion_c}
 \left|I_c(\varepsilon) - c(\lambda + \rho) \int_{S^*\mathbf M} vv^*\intd \mu_{\tu{L}} 
 - \sum_{\ell=1}^{\ell_0} \alpha_\ell \varepsilon^{\beta_\ell}\right|
  = \mc{O}(\varepsilon),
 \end{equation}
 and for some other set of coefficients $\alpha'_\ell$,
\begin{equation}\label{eq:pairing_asympt_expansion_0}
 \left|I_0(\varepsilon) - \sum_{\ell=1}^{\ell_0} \alpha'_\ell \varepsilon^{\beta_\ell}\right|
  = \mc{O}(\varepsilon). 
 \end{equation}
\end{lem}

\begin{rem}
 As we know that $I=I_c(\varepsilon)+I_0(\varepsilon)$ is independent of
 $\varepsilon$, we directly conclude that for all $\ell$ with 
 $\tu{Re}(\beta_\ell)< 0$ we have $\alpha_\ell=-\alpha'_\ell$
 and Theorem~\ref{thm:pairing_formula} is a direct consequence
 of Lemma~\ref{lem:pairing_asympt_expansion}.
\end{rem}

An important tool in the proof of the expansions \eqref{eq:pairing_asympt_expansion_c} and \eqref{eq:pairing_asympt_expansion_0}
will be the following differential 
operator on $\bN$
\[
 (L f)(\bn) := \frac{d}{d\tau}_{|\tau=0} f(e^{-\tau \cH}\bn e^{\tau \cH}),~~~f\in C^\infty(\bN),
\]
 which satisfies the relation  
$\varepsilon\partial_{\varepsilon}\chi_\varepsilon =  L\chi_\varepsilon$. 
Note that for $U_1\in \g_{-\alpha}, U_2\in\g_{-2\alpha}$ we have
\begin{equation}
 \label{eq:L_Euler}
 Lf(\exp (U_1+U_2)) = \frac{d}{d\tau}_{|\tau=0} f\circ\exp(e^\tau U_1 + e^{2\tau} U_2).
\end{equation}
Thus, the operator $L$ corresponds to a special linear combination of Euler
operators on the two different root spaces.

As $L$ commutes with the $M$-conjugation on $\bN$ we can lift it to a 
differential operator $\til L$ with smooth coefficients on 
$\Gamma\backslash G\times_M\bN$ in the following way:
\[
 \til L f([g,\bn]) := \frac{d}{d\tau}_{|\tau=0} f\left(\left[\Gamma g\exp\left(\tau \check H_0\right), e^{-\tau\cH}\bn e^{\tau\cH}\right]\right).
\]
Note that $\til L$ is not just the differential operator that differentiates 
 along the fibers of $\Gamma\backslash G\times_M\bN$, but it is twisted with a derivative 
along the geodesic flow on the base space $ \Gamma \backslash G/M$. Nevertheless, 
we still have the property 
$\varepsilon\partial_\varepsilon \chi_\varepsilon = \til L\chi_\varepsilon$,
where $\chi_\varepsilon$ is now understood as a function on 
$\Gamma\backslash G\times_M\bN$. The twist is crucial for the following lemma.

\begin{lem}\label{lem:smooth_extension}
 The differential operator $(\mathcal A^{-1})_*\til L$ on the open dense set $S^2_\Delta\mathbf M\subset S^2\mathbf M$
 extends to a first order differential operator on $S^2\mathbf M$ with smooth coefficients. 
\end{lem}

By abuse of notation we will denote this extended operator on $S^2\mathbf M$
again by $\til L$.

\begin{proof}
Just as we did for the definition of $\mc{A}$, we 
 lift everything to $\hspc$. By definition the 
 vector field $\mathcal A^{-1}_*\til L$ is smooth on 
 $S^2_\Delta \hspc$,  so we have to study its behavior 
 near the  antidiagonal $\{\eta_-+\eta_+=0\}\subset 
 S^2\hspc$, and we need some appropriate coordinates
 for this neighborhood. By the $NAK$-decomposition
 we can identify 
 $S^2\hspc \cong G/K\times K/M\times K/M \cong
 G/K\times G/P\times G/P $. The 
 antidiagonal consists of the points 
 $(naK, kP, kwP)$, $na\in NA, k\in K$, where again $w$ 
 is some representative of the nontrivial
 Weyl group element.  
 By the Bruhat decomposition $(naK, kP, kw\bn_\Delta P)$ with 
 $na\in NA, k\in K, \bn_\Delta\in\bN$ parametrizes an open
 neighborhood of the antidiagonal and the antidiagonal 
 corresponds to $\bn_\Delta=e$. We now want to compute 
 how these coordinates are transformed for $\bn_\Delta \neq e$
 under the diffeomorphism $\mathcal A$.
 
 The crucial point is that we need an explicit expression for 
 $\mathcal G^{-1}$. This can also be achieved by the 
 Bruhat decomposition which for real rank
 one implies that $\bN P \subset G$ is open and 
 dense. Given $g\in\bN P$ we define
 $\overline{n}_{\bN P}(g)\in\bN$ to be the unique element
 such that $gP = \overline{n}_{\bN P}(g)P$. Then we can write
 \[
  \mathcal G^{-1}(g_1wP, g_2P) = g_1\overline{n}_{\bN P}(g_1^{-1}g_2)MA.
 \]
 Putting everything together we obtain 
 \[\begin{split}
  \mathcal A(naK, kP, kw\bn_\Delta P) = & 
  \psi^{-1}\mc{G}^{-1}(naK,nakwP,nakw\bn_{\Delta}P)\\
  =& \psi^{-1}(naK, nak\Omega(\bn_{\Delta})MA)
= [nak\Omega(\bn_\Delta),
 \Omega(\bn_\Delta)^{-1}],
\end{split} \]
 where
 \[
  \Omega:\Abb{\bN\setminus\{e\}}{\bN\setminus\{e\}}
  {\bn}{\overline{n}_{\bN P}(w\bn))}.
 \]
 The inverse transformation reads
 \begin{equation}\label{eq:pairing_A_inv}
  \mathcal A^ {-1}([g, \bn]) = (g\bn K, k_{NAK}(g\bn) P, k_{NAK}(g\bn)w
  \Omega^{-1}(\bn^{-1})P).
 \end{equation}
 
Now the crucial observation for the smoothness of $\til L$ is that 
for any $\bn\in\bN$
\[
we^{-\tau\cH}\bn e^{\tau\cH}P=e^{\tau\cH}w\bn P=e^{\tau\cH}\bn_{\bN P}(w\bn) P= e^{\tau\cH}\bn_{\bN P}(w\bn)e^{-\tau \cH} P, 
\]
so that one has 
$\Omega(e^{-\tau\cH}\bn e^{\tau\cH})=e^{\tau\cH}\Omega(\bn) e^{-\tau\cH}$  and $\Omega_*L=-L$.

Using this identity and  \eqref{eq:pairing_A_inv} we
compute the transformation of the global vector field:
\begin{align}
\nonumber
 (\mathcal A^{-1}_*\til L) f&([naK, kP, kw\bn_\Delta P]) = 
 \frac{d}{d\tau}_{|\tau=0}f(\mathcal A^{-1}[nak\Omega(\bn_\Delta) e^{\tau\cH}, 
 e^{-\tau\cH}\Omega(\bn_\Delta)^{-1} e^{\tau \cH}])\\
  &= \frac{d}{d\tau}_{|\tau=0} f(nake^{\tau\check H_0}K,
 k_{NAK}(ke^{\tau\cH})P, k_{NAK}(ke^{\tau\cH})we^{\tau\cH} \bn_\Delta e^{-\tau\cH}P),
 \label{eq:L_near_diag}
\end{align}
which clearly is smooth in a neighborhood of the diagonal.\footnote{Without the 
derivative along the geodesics in $\til L$ 
there would occur a derivative $\frac{d}{d\tau}_{|\tau=0} nake^{-\tau\cH}\Omega(\bn)
e^{\tau\cH}$ for $\mathcal A^{-1}_*L$. But near the diagonal in the limit $\bn\to e$, $\Omega(\bn)\to\infty$, so this derivative would diverge. This shows the necessity of the precise choice of $\til L$.}
\end{proof}

Using the smooth vector field $\til L$ on $S^2\mathbf M$ 
we get the identity $\varepsilon\partial_\varepsilon \chi_\varepsilon =  -\til L(1-\chi_\varepsilon)$ (for the last equality note that $\til L1=0$). Now
 integration by parts yields 
\begin{eqnarray}
\varepsilon\partial_\epsilon I_0(\varepsilon) &=& \int_{S^2\mathbf M} (\til L^*(vv^*))(1-\chi_\varepsilon)
 \intd x\intd \eta_-\intd \eta_+ \label{eq:partical_int}\\
%  =-\int_{S^2\mathbf M} \til L^*(vv^*)\chi_\varepsilon \intd x\intd \eta_-\intd \eta_+, \\
\varepsilon\partial_\epsilon I_c(\varepsilon) &=& \int_{S^2\mathbf M} (\til L^*(vv^*))\chi_\varepsilon
 \intd x\intd \eta_-\intd \eta_+, \nonumber
\end{eqnarray}
where $\til L^*=-\til{L}-\tu{div}_{\intd x\intd\eta_-\intd\eta_+}(\til L)$ 
is the $L^2$ adjoint on $S^2\mathbf M$. By virtue of Lemma~\ref{lem:pairing_coordinate}
we obtain the alternative expression $\varepsilon\partial_\epsilon I_c(\varepsilon) =
\int_{S\mathbf M} vv^*\intd\mu_{\tu{L}} \cdot\int_{\bN} (L^* e^{-2(\rho+\lambda)H(\bn)})
\chi_\varepsilon(\bn)\intd \bn$, where $L^*=-L-\tu{div}_{\intd\bn}(L)$ is the $L^2$-adjoint
on $\bN$. 

The following lemma implies that $L$ is suitable for a
regularization at infinity in $\bn$ and $\til L$ is suitable for a regularization 
near the antidiagonal in $S^2M$. In order to formulate it, let us 
define $\langle \bn\rangle:= \|U_1\|^2+\|U_2\|$ where $ \bn=\exp(U_1+U_2)$,
$U_1\in\g_{-\alpha_0}, U_2\in\g_{-2\alpha_0}$.

\begin{lem}\label{lem:L-regularizing}
 There are functions $\mathcal V_{\ell}=\mathcal V_{\ell,\lambda} \in C^\infty(\bN)$ 
 depending polynomially on $\lambda$  as well as polynomials 
 $\beta_\ell(\lambda), \ell=1,2,\ldots$ 
 that do not vanish for $\lambda\in \C\setminus\{-\rho - \N_0\alpha_0\}$ such that
 \begin{equation}\label{eq:L-regularizing}
  (L^*-\beta_\ell(\lambda)) \mathcal V_{\ell-1}(\bn)e^{-(2\rho+\lambda)H(\bn)} = 
  \mathcal V_\ell(\bn)e^{-(2\rho+\lambda)H(\bn)},
 \end{equation}
 where we set $\mathcal V_0(\bn)=1$ and for fixed $\lambda$ the functions $\mathcal V_\ell$ satisfy the uniform pointwise 
 estimate $\mc{V}_\ell= \mc O(\langle \bn\rangle^{-\ell})$
 for $\langle \bn\rangle \to\infty$.
 
 Furthermore $\mathcal V_\ell \circ \mathcal A\in C^\infty(S^2_\Delta \mathbf M)$
 extends to a smooth function $\til{\mathcal  V}_\ell \in C^\infty(S^2\mathbf M)$
 that vanishes of order $\mc{O}(\langle\bn_\Delta\rangle^\ell)$ at the antidiagonal.\footnote{Recall that the coordinates $\bn_\Delta$ near the antidiagonal have been introduced in the proof of Lemma~\ref{lem:smooth_extension}.} Moreover, if $\lambda\in\C\setminus\{-\rho - \N_0\alpha_0\}$
 is a  Ruelle resonance in the first band and $v,v^*$ are corresponding 
 resonant/co-resonant states, then 
 \begin{equation}\label{eq:Ltilde-regularizing}
(\til L^* - \beta_\ell(\lambda) )\til{\mathcal V}_{\ell-1}vv^* = 
\til{\mathcal V}_\ell vv^*  
 \end{equation}
\end{lem}

Before we prove Lemma~\ref{lem:L-regularizing} let us show how it implies 
Lemma~\ref{lem:pairing_asympt_expansion}:
\begin{proof}[Proof of Lemma~\ref{lem:pairing_asympt_expansion}]
From \eqref{eq:partical_int} and \eqref{eq:Ltilde-regularizing} we get  
\[
 \prod_{\ell=0}^{\ell_0-1} 
 (\varepsilon\partial_\varepsilon - \beta_{\ell_0-\ell}(\lambda))I_0(\varepsilon) 
 = \int_{S^2\mathbf M} (vv^*) \cdot (\til{\mathcal V}_{\ell_0}(1-\chi_\varepsilon)) 
 \intd x\intd\eta_-\intd\eta_+.
\]
Recall that $\til{\mathcal V}_{\ell_0}$ vanishes to the order $\ell_0$ at the antidiagonal and $(1-\chi_\varepsilon)$ (identified with its pull-back $\mc{A}^*(1-\chi_\varepsilon)$) is supported in an $\varepsilon$-neighbourhood of the antidiagonal, thus 
$\|(\til{\mathcal  V}_{\ell_0}(1-\chi_\varepsilon))\|_{C(S^2{\bf M})} = \mc{O}(\varepsilon^{\ell_0})$. Furthermore we have $\|(1-\chi_\varepsilon)\|_{C^k(S{\bf M})} =\mathcal O(\varepsilon^{-k})$, and because $\til{\mathcal V}_{\ell_0}$ is vanishing to order $\ell_0$ at the antidiagonal and is also $C^\infty$, the Taylor formula implies that for any smooth $k$-th order differential operator $P$ on $S^2{\bf M}$ and $k<\ell_0$ the function $P\til{\mathcal V}_{\ell_0}$ vanishes to order $\ell_0-k$ at the antidiagonal. Putting everything together we get 
$\|(\til{\mathcal  V}_{\ell_0}(1-\chi_\varepsilon))\|_{C^{\ell_0 - 1}(S^2{\bf M})} = \mc{O}(\varepsilon)$. 
As $vv^*$ is a distribution of finite order we get that $\prod_{\ell=1}^{\ell_0} (\varepsilon\partial_\varepsilon - \beta_{\ell_0-\ell}(\lambda))I_0(\varepsilon) 
 = \mc{O}(\varepsilon)$ for sufficiently
large $\ell_0$ depending on the order of $vv^*$. Consequently, there exist $\alpha_\ell\in\C$ such that
\[
  \left|I_0(\varepsilon) - \sum_{\ell=1}^{\ell_0} \alpha_\ell \varepsilon^{\beta_\ell(\lambda)}\right|
  = \mc{O}(\varepsilon).
\]

Next let us  define for any $\lambda\in\C$ 
\[
 c_\varepsilon(\lambda) := \int_{\bN} e^{-(2\rho+\lambda)H(\bn)} \chi_\varepsilon(\bn)
 \intd\bn,
\]
which, by \cite[Ch IV.6]{Hel84}, converges for 
$\tu{Re}(\lambda)>-\rho$ to $c(\lambda+\rho)$. While the meromorphic 
continuation of the Harish-Chandra c-function is well-known, we also need a precise asymptotic expansion in powers of $\varepsilon$.
To that end we consider \eqref{eq:L-regularizing} which implies
\begin{equation}\label{eq:c-fct_regularizing}
 \prod_{\ell=0}^{\ell_0-1} 
 (\varepsilon\partial_\varepsilon - \beta_{\ell_0-\ell}(\lambda))c_\varepsilon(\lambda) 
 = \int_{\bN} e^{-(2\rho+\lambda)H(\bn)} \mathcal V_{\ell_0}(\bn)\chi_\varepsilon(\bn)\intd\bn.
\end{equation}
From \cite[Thm IX.3.8]{Hel78} we know that
$e^{-(2\rho+\lambda)H(\bn)}  = \mc O(\langle\bn\rangle^{-(2\rho+\tu{Re}(\lambda))/\|\alpha_0\|})$. 
Thus, if we fix $C>0$, then for sufficiently large $\ell_0$ the right side of \eqref{eq:c-fct_regularizing} converges to a function $h(\lambda)$
holomorphic on $\{\tu{Re}(\lambda) >-C\}$:
\[
\prod_{\ell=0}^{\ell_0-1} 
 (\varepsilon\partial_\varepsilon - \beta_{\ell_0-\ell}(\lambda))c_\varepsilon(\lambda) 
 = h(\lambda) +\mc{O}(\varepsilon),
\]
and consequently there exist $\alpha'_\ell\in\C$ such that
\[
\left|c_\varepsilon(\lambda) - 
\frac{h(\lambda)}{\prod_{\ell = 1}^{\ell_0} \beta_\ell(\lambda)}
-\sum_{\ell=1}^{\ell_0} \alpha'_\ell \varepsilon^{\beta_\ell(\lambda)}\right| = \mc{O}(\varepsilon).
\]
As we know that the integral defining $c_\varepsilon(\lambda)$ converges for $\tu{Re}(\lambda)$ large 
enough, we get by uniqueness of meromorphic continuation that 
$\frac{h(\lambda)}{\prod_{\ell = 1}^{\ell_0} \beta_\ell(\lambda)} = c(\lambda+\rho)$.
This yields \eqref{eq:pairing_asympt_expansion_c}.
\end{proof}

The only remaining task is to prove Lemma~\ref{lem:L-regularizing}

\begin{proof}[Proof of Lemma~\ref{lem:L-regularizing}]
In order to simplify notation let us define $\mc{N}(\bn):= e^{\alpha_0 H(\bn)}$, $\check \rho:=\rho/\|\alpha_0\|$ and $\check\lambda:=\lambda/\|\alpha_0\|$.
For $\bn=\exp(U_1+U_2)$, we then have the explicit expression 
 \cite[Thm IX.3.8]{Hel78}
 \begin{equation}
  \label{eq:X_coordinates}
  \mc{N}(\bn) = \left[(1+c\|U_1\|^2)^2 + 4c\|U_2\|^2\right]^{1/2},
 \end{equation}
 where $c=(4m_{\alpha_0} + 16m_{2\alpha_0})^{-1}$. Thus, we have 
 $\mc{N}(\bn) \asymp \langle\bn\rangle$ for $\bn\to\infty$.
 
 We first want to prove the existence of $\beta_\ell$ and $\mathcal V_\ell$ fullfilling \eqref{eq:L-regularizing} by induction. In order to study \eqref{eq:L-regularizing} for $\ell=1$ we 
 compute 
 \[
  L^*\mc{N}^{-(2\check\rho + \check\lambda)} = \left((2\check \rho+\check \lambda)
  \frac{L\mc{N}}{\mc{N}} -
  2\check \rho\right)\mc{N}^{-(2\check\rho +\check \lambda)}.
 \]
 Here we used the fact that $(\log)_*\intd \bn$ is the Lebesgue measure on $\n_-$ and from the expression of $L$ as an Euler operator \eqref{eq:L_Euler}, we get $\tu{div}_{\intd\bn}(L) = 2\check\rho$.
We are thus lead to 
 study the properties of $(L\mc{N})/\mc{N}$. Using \eqref{eq:X_coordinates}
 we get 
 \begin{equation}\label{LN/N}
  \frac{L\mc{N}}{\mc{N}} = 2-\frac{\mc{Q}}{\mc{N}^2}, 
 \end{equation}
where $\mc{Q}(\bn) = 2+2c\|U_1\|^2$.
Putting everything together we get 
\begin{equation}
 \label{eq:lstarN}
 L^*\mc{N}^{-(2\check\rho+\check\lambda)} = \left(
 2(\check\rho + \check\lambda) - (2\check\rho + \check\lambda) \frac{\mc{Q}}{\mc{N}^ 2}
 \right)\mc{N}^{-(2\check\rho+\check\lambda)}.
\end{equation}
We note that $\frac{\mc{Q}}{\mc{N}^ 2} = \mc{O}(\langle \bn\rangle^{-1})$ for $\bn\to\infty$. 
Thus, we have found $\beta_1(\lambda) = 2(\check\rho + \check\lambda)$ and 
$\mathcal V_{1,\lambda}=-(2\check \rho + \check\lambda)\frac{\mc{Q}}{\mc{N}^ 2}$.
In order to examine \eqref{eq:L-regularizing} for $\ell=2$ we additionally need to compute $L\frac{\mathcal Q}{\mathcal N^2}$. A straightforward computation yields
\[
 (L+2)\frac{\mc{Q}}{\mc{N}^2} = 
 2\left(\frac{\mc{Q}}{\mc{N}^2}\right)^2 -4\frac{1}{\mc{N}^2}.
\]
First of all, both terms on the right hand side are $\mc{O}(\langle\bn\rangle^{-2})$. This observation directly implies that $\mathcal V_2$ is $\mathcal O(\langle\bn\rangle^{-2})$. Secondly on the right hand side, only powers of $\frac{\mc{Q}}{\mc{N}^2}$ and $\frac{1}{\mc{N}^2}$ appear. This allows for a straightforward induction procedure because all appearing $\mc V_\ell$ are some homogeneous polynomials of degree $\ell$ in the variable $(\mc{Q}/\mc{N}^2,1/\mc{N}^2)$ with coefficients polynomial in $\check\la$. Furthermore one checks that $\beta_{\ell}(\la)=2(\check\rho + \check\lambda+\ell-1)$. This establishes the existence of $\beta_\ell, \mc V_\ell$ fulfilling \eqref{eq:L-regularizing}.

Let us now turn to \eqref{eq:Ltilde-regularizing}. We note that, as $\til L$ is a smooth 
 differential operator, we can compute its divergence on $S^2_\Delta\mathbf M$.
 Using \eqref{eq:push_forward_measure} as well as the fact that $\intd \mu_{\tu{L}}$
 is preserved by the geodesic flow, we compute
 \[
  \tu{div}(\til L) = -2\check\rho\frac{L\mc{N}}{\mc{N}} + \tu{div}_{\intd \bn}(L).
 \]
 Furthermore, using $(\mathcal A_*(vv^*))([g,\bn]) = (vv^ *)(gM) \mc{N}^{-\check \lambda}(\bn)$
 as well as the description of $\til L$ near the antidiagonal \eqref{eq:L_near_diag},  we obtain
 $\til L (vv^*) = -\check\lambda \frac{L\mc{N}}{\mc{N}} (vv^*)$, whence
 \[
  \til L^*(vv^*) = \left((2\check \rho+\check \lambda)\frac{L\mc{N}}{\mc{N}} -
  \tu{div}_{\intd \bn}(L)\right)(vv^*).
 \]
 Comparing this equation to \eqref{eq:lstarN} and recalling $\tu{div}_{\intd \bn}(L)=2\check\rho$ we directly see that we are led to the same recursion as for \eqref{eq:L-regularizing} and consequently $\til{\mc V}_\ell = \mc V_\ell\circ\mc A$ on $S^2_\Delta{\bf M}$. 
The existence of a smooth extension to $S^2\mathbf M$ is obvious
by the following global argument: 
$\til L$ is a smooth vector field 
thus its divergence is a smooth function, and all 
$\til{\mathcal V}_{\ell}$ are built out of derivatives of this
divergence. The smoothness can, however, also be seen in the coordinates
$n,a, k,\bn_\Delta$ around the antidiagonal, which were introduced in the proof
of Lemma~\ref{lem:smooth_extension}. In these coordinates
\[
 \mc{N} = e^{\alpha_0 H(\bn_{\bN P}(w\bn_\Delta))} = 
 e^{\alpha_0 (H(\bn_\Delta) - B(w\bn_\Delta))}
\]
where $B=\log(a_{\bN MAN})$. Now for $\bn_\Delta = \exp(V_1+V_2)$,
$V_1\in\g_{-\alpha_0}, V_2\in\g_{-2\alpha_0}$ \cite[Thm IX.3.8]{Hel78}
gives us
\[
 e^{\alpha_0 B(w\bn_\Delta))} = [c^2\|V_1\|^4 + 4c\|V_2\|^2]^{\frac{1}{2}}
\]
and with a completely analogous computation as above we check the 
vanishing of the iteratively defined $\til{\mathcal V}_{\ell}$.
\end{proof}

\section{Equidistribution for Ruelle resonant states}
Let us finally draw the desired conclusions concerning
the high frequency limits of invariant Ruelle distributions from the explicit relations between
Ruelle resonant states and Patterson-Sullivan distributions.

For $\mc{M} = \Gamma\backslash  G/M=S{\bf M}$ a number $\la=-\rho+\mu \in \C$ is a Ruelle
resonance of the first band if and only if the complex conjugate $\bbar{\la}$
is a Ruelle resonance of the first band as well. This follows
from the fact that the generating vector field $X$ commutes
with complex conjugation. By this symmetry of the spectrum it is enough to consider
first band resonances with $\Im(\la)\geq 0$, and since we are interested in high frequency limits, we take $\Im(\mu)>0$. We thus denote by $\la_n=-\rho+ir_n\in \C$ a sequence
of Ruelle resonances in the first band with $r_n >0$ and $r_{n+1}> r_n$. We do not want to
repeat them according to multiplicity but rather work with the multiplicities
$m_n := \dim(\tu{Res}^0_{X}(\la_n))$. Now to any subsequence
$(r_{k_n})_{n>0} \subset (r_{n})_{n>0}$ we can associate the sequence of
invariant Ruelle distributions $\mathcal T_{\la_{k_n}} \in \mc{D}'(\mc{M})$
and we study the weak limits of these sequences.

Then we obtain the following reformulation of Theorem~\ref{thm: Main Theorem}.

\begin{thm}\label{thm:hyperclassical}
 Let $\mc{M} = \Gamma\backslash G/M=S{\bf M}$ the unit tangent  bundle of a
compact locally Riemannian symmetric space of rank one and $\intd\mu_{\tu{L}}$ the Liouville measure.
Then there is a subsequence $(r_{k_n})_{n>0} \subset (r_{n})_{n>0}$
of density one, i.e. with
\[
 \lim_{N\to\infty}\frac{\sum_{k_n<N} m_{k_n}}{\sum_{n<N} m_{n}}=1,
\]
such that $\mathcal T_{r_{k_n}}$ converges weakly to $\intd \mu_{\tu{L}}$ in 
$\mc{D}'(\mathcal M)$ as $n\to\infty$.
\end{thm}

\begin{proof}
According to Theorem~\ref{thm:intertwining} we can associate to each Ruelle
resonance $\la_n=-\rho+ir_n$ a spectral parameter $\mu_n:=ir_n \in \a^*_\C$ of the Laplacian  such that $\dim(\tu{Eig}_{\Delta_{\bf M}}(\mu_n))=\dim(\tu{Res}^{0}_{X}(\la_n)) 
= m_n$.
For each of these $\Delta_{\bf M}$-eigenspaces we choose an orthonormal basis $\varphi_{n,l}$
of real-valued functions, where $l=1,\ldots,m_n$. Let $B= \{\varphi_{n,l}\}$ be the set of all these basis
vectors. Then the quantum ergodicity theorem of Shnirelman-Zelditch-Colin de Verdière
\cite{Shn74, CdV85, Zel87} implies the existence of a density one subsequence such that
the Wigner distributions $\mathcal W_{\phi, n}$ (see Appendix~\ref{app:wigner} for a precise definition) converge towards the Liouville measure. More precisely, we can
split $B$ into the disjoint union $B=B^\tu{good}\cup B^\tu{bad}$ such that for any sequence in
$B^\tu{good}$ the Wigner distributions converge towards the Liouville measure and
\[
 q_N:= \frac{\sum_{n\leq N} m_n^\tu{bad} }{\sum_{n\leq N} m_n}\to 0,~\tu{as }N\to\infty,
\]
where $m_n^\tu{bad}:=\# (B^\tu{bad}\cap\tu{Eig}_{\Delta_{\bf M}}(ir_n))$. In
order to obtain the subsequence $r_{k_n}$ for the convergence of the invariant Ruelle
distributions, we remove from the full sequence $(r_n)_{n>0}$ all elements for which
$m_n^\tu{bad}/m_n \geq \eps_n$, where $\eps_n:=(\sup_{k\geq n}q_k)^{1/2}$ is a decreasing sequence converging to $0$. Thus, we obtain a subsequence $(r_{k_n})_{n\geq0}$
for which we have 
\[ \frac{\sum_{n< N} m_{n}-\sum_{k_n\leq N}m_{k_n}}{\sum_{n\leq N}m_n}
\leq \frac{q_N}{\eps_N}\leq \sqrt{q_N}\]
or, equivalently, 
\[ \frac{\sum_{k_n\leq N}m_{k_n}}{\sum_{n\leq N}m_n}\geq 1-\sqrt{q_N}\to 1.\]
Let us finally show the convergence of the subsequence of invariant Ruelle distributions:
we fix a Ruelle resonance $\la_{k_n}$. Then, using the basis $\varphi_{k_n,l}$ and the isomorphism
$\mathcal I_-(ir_{k_n})$,
we define a basis $u_l:= \mathcal I_-(\mu_{k_n})(\varphi_{k_n, l}) \in \tu{Res}_{X}(\la_{k_n})$.
Theorem~\ref{thm:intertwining}
implies that $u_l$ is a basis of $\tu{Res}^0_{X}(\la_n)$. In addition we define the basis
of co-resonant states 
\[u_l^*:= \frac{1}{\operatorname{PS}^\Gamma_
{\varphi_{k_n,l},\varphi_{k_n,l}}(1)}\mathcal I_+(\mu_{k_n})(\varphi_{k_n,l})\in
\tu{Res}_{X^*}(\la_{k_n}).\]
The pairing of co-resonant states and resonant states is simply given by the
pairing of distributions with disjoint wavefront sets, which in turn is the product of the distributions paired with $1$. We recall the pairing formula of Theorem~\ref{thm:pairing_formula}: let $\lambda \in \a_\C^*$ and for $v\in \tu{Res}_{X}^0(\lambda)$, 
$v^*\in \tu{Res}_{X^*}^0(\lambda)$. Then one has the identity
\begin{equation} \label{eq:pairing}
\langle\pi_{0*} v^*, \overline{\pi_{0*}v} \rangle_{L^2(\Gamma\backslash G/K)} = c(\lambda + \rho)\cdot (v v^*)[1_{S\mathbf M},]
\end{equation}
where $c(\mu) = \int_N e^{-(\mu + \rho)H(nw)}\intd n$ is the Harish-Chandra c-function. 

Thus, the pairing formula together with the chosen normalization of $u_i^*$ implies that $u_i^*[u_j]=\delta_{ij}$,
which means, that the basis $u_i^*$ is dual to the chosen basis $u_i$. By Theorem~\ref{thm:jordan} we know that there are no Jordan blocks
at the spectral parameter $\la_{k_n}$, and thus the spectral projector can be written
as $\hat \Pi_{k_n} = \sum_{l=1}^{m_{k_n}} u_l\otimes u_l^*$. Now for $f\in C^{\infty}(\mathcal M)$
we obtain that
\[
 \mathcal T_{k_n}[f] = \frac{1}{m_{k_n}} \sum_{l=1}^{m_{k_n}} u_l^*[f\cdot u_l]
 = \frac{1}{m_{k_n}}\sum_{l=1}^{m_{k_n}} \frac{\operatorname{PS}_{\varphi_{k_n,l}, \varphi_{k_n,l}}[f]}{\operatorname{PS}_{\varphi_{k_n,l}, \varphi_{k_n,l}}[1]}
\]
If $\til f\in C_c^\infty(T\bf{M})$ is an arbitrary compactly
supported function such that $\til f|_{S{\bf M}}=f$, then
by \cite{HHS12} 
\[
 \mathcal T_{k_n}[f] =\frac{1}{m_{k_n}} \sum_{l=1}^{m_{k_n}} \mathcal W_{\varphi_{k_n,l}}[\til f]
 + \mathcal O\left(\frac{1}{r_{k_n}}\right).
\]
Note that there can still appear some Wigner distributions from the bad eigenfunctions
$\varphi_{k_n}$ that do not converge towards the Liouville measure. However, by the
choice of the subsequence, there are only few of them and we can write
\[
 \mathcal T_{k_n}[f] =\frac{1}{m_{k_n}}\left( \sum_{\varphi_{k_n, l} \in B^\tu{good}\cap \tu{Eig}_{\Delta_\Gamma}(\mu_{k_n})} \mathcal W_{\varphi_{k_n,l}}[\til f]
 + \sum_{\varphi_{k_n, l} \in B^\tu{bad}\cap \tu{Eig}_{\Delta_\Gamma}(\mu_{k_n})} \mathcal W_{\varphi_{k_n,l}}[\til f] \right)
 + \mathcal O\left(\frac{1}{r_{k_n}}\right).
\]
The Wigner distributions in the first sum converge towards the Liouville measure
and the contributions of the second sum can be bounded with a standard $L^2$ estimate
(see e.g. \cite[Theorem 5.1]{Zw12}) by
\[
 |\mathcal W_{\varphi_{k_n,l}}[\til f]|\leq C\sup|\til f| + \mathcal O(1/r_{k_n}).
\]
Thus
\[
\mathcal T_{k_n}[f] =\frac{1}{m_{k_n}}\left( \sum_{\varphi_{k_n, l} \in B^\tu{good}\cap \tu{Eig}_{\Delta_\Gamma}(\mu_{k_n})} \mathcal W_{\varphi_{k_n,l}}[\til f]
 +  \right) + \frac{m_{k_n}^\tu{bad}}{m_{k_n}}C\sup|\til f|
 + \mathcal O\left(\frac{1}{r_{k_n}}\right).
\]
But as we have chosen the subsequence such that
$\frac{m_{k_n}^\tu{bad}}{m_{k_n}} \to 0$, we conclude that
in the limit $n\to\infty$ only the ``good'' Wigner distributions contribute and
we obtain
\[
\mathcal T_{k_n}[f] \to \int_{S\mathbf M} f \intd \mu_{\tu{L}},
\]
concluding the proof.
\end{proof}

\appendix
\section{Wigner distributions}\label{app:wigner}
\begin{Def}
 \label{def:wigner}
 Let $r>0, \mu=ir$ and $\varphi \in \tu{Eig}_{\Delta_{\mathbf M}} (\mu)$ be an $L^2$-normalized eigenfunction
 of the Laplacian. Then we define the associated \emph{
 Wigner-distribution} \footnote{Although they are commonly called distributions, strictly speaking, Wigner distributions are by definition generalized densities.
 However, the space $C^{-\infty}(T^*\mathbf M)$ of generalized densities can be identified with the space $\mathcal D'(T^*\mathbf M)$ of distributions via the canonical measure on $T^*\mathbf M$.}
 $\mathcal W_\varphi\in C^{-\infty}(T^*\mathbf M)$
 as follows:
 \[
  \mathcal W_\varphi:\Abb{C_c^\infty (T^*\mathbf M)}{\C}{a}
  {\langle \tu{Op}_{1/|\mu|} (a) \varphi, \varphi\rangle_{L^2}}
 \]
Here $\langle\bullet,\bullet\rangle_{L^2}$ is the $L^2$-scalar product on $\mathbf M$ and
for any $\hbar >0$, $\tu{Op}^w_\hbar (a)$ is a bounded operator on $L^2(\mathbf M)$ obtained by a semiclassical Weyl quantization (see e.g. \cite{Zw12},\cite[Appendix E]{DZ18}).
\end{Def}
In quantum ergodicity one is then interested in understanding the weak limits of
these generalized densities. They have the following important properties.
\begin{prop}
 \label{prop:semiclassical_measures}
 Let $r_n>0$ be the positive real numbers, such that 
 $\tu{Eig}_{\Delta_{\mathbf M}} (ir_n)\neq 0$,
 where we repeat $r_n$ according to the multiplicity  (dimension of
 $\tu{Eig}_{\Delta_{\mathbf M}} (ir_n)$), and let $\psi_n\in \tu{Eig}_{\Delta_{\mathbf M}} (ir_n)$ be an associated $L^2$-normalized eigenfunction. If there is
 $\mu\in C^{-\infty}(T^*\mathbf M)$ as well as a subsequence $r_{n_k}$
 such that $\mathcal W_{\psi_{n_k}}\to \mu$ weakly, then $\mu$ is a positive Radon
 measure that is supported on $S^*\mathbf M$ and we call it a
 \emph{semiclassical measure}.
\end{prop}
\begin{proof}
 The fact that the semiclassical limits are positive Radon measures follows from
 a standard compactness argument (see e.g. \cite[Chapter 5]{Zw12}). For the support
 on $S^*\mathbf M$ see e.g. \cite[Chapter 15]{Zw12}.
\end{proof}
\section{Transformation formulas}

In this appendix we prove the transformation formulas used in the proof of Lemma~\ref{lem:pairing_coordinate}.
 
Let $\intd(gMA)$ be the $G$-invariant measure on $G/MA$ which is uniquely 
determined\footnote{Recall that the measures $\intd g,\intd a\intd m$ have been fixed 
in Section~\ref{sec:rieman_sym_space}.}
by the condition that for all $f\in C_c^\infty(G)$ we have
\begin{equation}\label{eq:dgMA}
\int_G f(g)\intd g
=\int_{G/MA} \int_M \int_A f(gma)\intd a\intd m\intd(gMA).
\end{equation}
The following lemma gives two helpful expressions for $\intd (gMA)$
\begin{lem}\label{lem:dgMA}
\begin{itemize}
\item[(i)]
There is a constant $c_1$ such that for all $f\in C_c(G/MA)$
\[
 \int_{G/MA} f(gMA) \intd(gMA) = c_1 \int_K\int_{N} f(knMA) \intd n\intd k
\]
\item[(i)]
There is a constant $c_2$ such that for all $f\in C_c(G/MA)$
\[
 \int_{G/MA} f(gMA) \intd(gMA) = c_2 \int_N \int_{\bN} f(n\bn MA) \intd \bn\intd n
\]
\end{itemize}
\end{lem}
\begin{proof}
 (i) Let $f\in C_c^\infty(G/MA)$ and fix $\chi\in C_c^\infty(A)$ such that 
 $\int_A\chi(a)\intd a =1$. Using the $KNA$ decomposition we can find a 
 function $\til f\in C_c^\infty(KN)^M$ such that 
 $f(gMA) = \til f(k_{KNA}(g)n_{KNA}(g))$. Then by \eqref{eq:dgMA}
 we and \cite[Cor.~I.5.3]{Hel84} we get 
 \begin{eqnarray*}
  \int_{G/M} f(gM) &=& \int_G \til f(k_{KNA}(g)n_{KNA}(g))\chi(a_{KNA}(g))\intd g\\
  &=&c_1\int_K\int_N\int_A \til f(kn)\chi(a)\intd a\intd n \intd k\\
  &=&c_1\int_K\int_N f(knMA)\intd n\intd k
 \end{eqnarray*}
 
(ii) The Bruhat decomposition of $G$ implies that 
the group multiplication induces a diffeomorphism 
$N\times \overline N\times M\times A\to N\overline NMA$
with $N\overline NMA\subseteq G$ open dense and complement of measure zero. 
It is thus enough to prove the equality for $f\in C_c^\infty(N\bN MA/MA)=C_c^\infty(N\bN)$.
The equality then follows from \cite[Prop.~I.5.21]{Hel84},  which implies
\begin{equation}\label{eq:Bruhat-integral}
\int_G f(g)\intd g
=c_2 \int_{N}\int_{\overline N}\int_M\int_A f(n\overline nma)\intd a\intd m\intd\overline n\intd n.
\end{equation}
\end{proof}

\begin{prop}\label{prop:Mass1}
There is a constant $c_{\mathcal G}$ such that
\[
\mathcal G_*\intd(gMA) = c_{\mathcal G}(d_0(kM, k'M))^2\, \big(\intd (kM)\otimes \intd (k'M)\big)|_{(\partial\hspc)^2_\Delta}.
\]
\end{prop}

\begin{proof}
Let $f\in  C^\infty_c((\partial\hspc)^2_\Delta)\subseteq C^\infty(\pl{\bf H}\times \pl{\bf H})$. Lemma~\ref{lem:dgMA}(i),
together with  \cite[Prop. 2.11]{HHS12} and \cite[Thm. 5.20]{Hel84}, shows that 
\[ 
\int_{(\pl{\bf H})^{(2)}_\Delta} f\,d_0^{-2}\, \mathcal G_*\intd(gMA)
=\int_{G/MA} f(gP,gwP)d_0(gP,gwP)^{-2}\intd(gMA)\]
is actually equal to 
\[
 c_1 \int_{K}\int_K f(kP,k'P)\intd k'\intd k=
 \int_{\pl{\bf H}\times \pl{\bf H}} f\, \intd(kM)\otimes \intd (k'M).
\]
\end{proof}

Recall the map 
$\psi: G\times_M \overline N\to G/K\times G/MA$ from \eqref{eq:def-psi} and note that 
$$\psi^{-1}(G/K\times N\overline NMA/MA)=(N\overline NMA\times \overline N)/M \cong N\overline NA\times \overline N.$$

Let $\intd (gK)$ be the $G$-invariant measure on $G/K$ normalized by the Killing metric on $G/K$. By our choices in Section~\ref{sec:rieman_sym_space} this coincides with the push-forward 
of the Haar measure on $G$ by the canonical projection $\mathrm{pr}_K: G\to G/K$. 
Further, let $\til{\mathrm{pr}}_M: G\times \bN\to G\times_M \bN$ be the canonical 
projection with respect to the $M$-action. % \eqref{eq:M-actionGtimesN}.

\begin{prop}\label{prop:Mass2} There is a constant $c_\psi>0$
\[
(\psi^{-1})_*\big(\intd(gK)\otimes \intd(g'MA)\big)=
c_\psi ({\til{\mathrm{pr}}_M})_*(\intd g \otimes \intd \bn).
\]
\end{prop}

\begin{proof}
For $h\in C^\infty_c(G/K\times N\overline NMA/MA)=C^\infty_c(G/K\times N\overline N)$ we  compute
\begin{eqnarray*}
&&\int_{G/K\times N\overline NMA/MA} h(gK,g'MA)\, \intd(gK)\otimes\intd(g'MA)\\
&=&c_1 \int_{G/K}\int_{\overline N}\int_{N} h(gK, n\overline nMA)\, \intd n\, \intd\overline n\,
\intd (gK) \\
&=&c_1c_2 \int_A\int_{N}\int_{\overline N}\int_{N} h(a n'K, n \overline nMA)\,\intd n\, 
\intd\overline n\, \intd n'\, \intd a\\
\end{eqnarray*}
using Lemma~\ref{lem:dgMA}(ii). 
Now, if $\til h\in C^\infty_c(N\overline NMA\times \overline N)^M=C^\infty_c(N\overline NA\times \overline N)$, we obtain
\begin{eqnarray*}
&&\int_{(N\overline NMA\times \overline N)/M} \til h\, (\psi^{-1})_*(\intd(gK)\otimes\intd(g'MA))\\
&=&\int_{G/K\times N\overline NMA/MA} (\til h\circ \psi^{-1})  \intd(gK)\otimes\intd(g'MA)\\
&=& c_1c_2\int_A\int_{N}\int_{\overline N}\int_{N} (\til h\circ \psi^{-1})(a n'K, n \overline nMA)\, \intd n\intd\overline n\intd n'\intd a\\
&=& c_1c_2\int_A\int_{N}\int_{\overline N}\int_{N} \til h\big(n\overline n\,a_{A\bN K}((n\overline n)^{-1}a n'), \overline n_{A\bN K}((n\overline n )^{-1}a n')\big)\intd n\intd \overline n\intd n'\intd a.
\end{eqnarray*}
We now use Fubinis theorem and arrange the integrals such that the $\intd n'$ integral is the inner integral. Then the invariance of $\intd\overline n'$ by left multiplication implies that 
the latter integral equals
\[
c_1c_2\int_{N}\int_\bN\int_A\int_{N} \til h\big(n\overline n\,a_{A\bN K}(\overline n^{-1}a n'), \overline n_{A\bN K}((\overline n^{-1}a n')\big) 
\intd n'\intd a\intd\overline n\intd n.
\]
Reordering the integrals such that the $\intd a$ integral is the interior integral 
and using the invariance of $\intd a$ we can transform our expression to
\[
c_1c_2\int_{N}\int_\bN\int_A\int_{N} \til h(n\overline n\,a, (a^{-1}\overline n^{-1}a)\,
\overline n_{A\bN K}(n'))\intd a\intd  n'\intd\overline n\intd n.
\]

Note that $N\to \bN,\  n'\mapsto w^{-1}n'w$ is measure preserving. Therefore we can rewrite this integral as 
\[
c_1c_2\int_{N}\int_A \int_\bN\int_\bN\til h(n\overline  n\,a,
(a^{-1}\overline n^{-1}a)\,
\overline n_{A\bN K}(w\overline n' w^{-1}))\intd \overline n'\intd \overline n\intd a\intd n. 
\]
Moreover we have that the map $\nu:\bN\to\bN, \bn\to \bn_{A\bN K}(w\bn w^{-1})$ has Jacobian equal to one (see Lemma~\ref{lem:Jac_nu}). Thus we obtain
\[
c_1c_2\int_{N}\int_A \int_\bN\int_\bN\til h(n\overline n\,a,
(a^{-1}\overline n^{-1}a)\, \overline n')\intd \overline n'\intd \overline n\intd a\intd n.
\]
Using again the Fubini trick, this time with the invariance of $\intd \bn'$ we obtain
\[
c_1c_2\int_{N}\int_A \int_\bN\int_\bN\til h(n\overline n\,a,\overline n')\intd \overline n'\intd \overline n\intd a\intd n
=c_1c_2\int_{G}\int_\bN\til h(g,\overline n')\intd n'\intd g,
\]
where we use \eqref{eq:Bruhat-integral} for the final equality. 
In summary we have shown
\[
c_1c_2\int_{(N\bN MA\times \bN)/M} \til h\,(\psi^{-1})_*(\intd(gK)\otimes\intd(g'MA)
=c_1c_2\int_{N\bN A}\int_\bN\til h\intd \overline n'\intd g.
\]
Since $\psi$ is a diffeomorphism so that the measure zero set $G\setminus(N\bN MA)$ gets mapped to a measure zero set, namely $(N\bN AM\times \bN)/M$, this  proves the claim.
\end{proof}
\begin{lem}\label{lem:Jac_nu}
 The diffeomorphism 
 \[
  \nu:\Abb{\bN}{\bN}{\bn}{\bn_{A\bN K}(w\bn w^{-1})}
 \]
 has Jacobi determinant $|\det D\nu|=1$.
\end{lem}
\begin{proof}
Note that  $\til \bn=\bn_{A\bN K}(w\bn w^{-1}) $ implies that there are $\til a\in A$ and $\til k\in K$ such that $w\bn w^{-1}=\til a\til \bn\til k$.
But then 
\[
w^{-1}\til \bn w =w^{-1}\til a^{-1}w \bn w^{-1} \til k^{-1}w 
\]

so that $\bn=\bn_{A\bN K}(w^{-1}\til \bn w)$. 
\begin{description}
\item[Case 1] Suppose that $w=w^{-1}$.  Then $n=\nu(\til n)$, i.e. we have that $\nu=\nu^{-1}$. As a consequence the Jacobian determinant is $1$. 
\item[Case 2] Suppose that one cannot choose $w$ such that $w=w^{-1}$. Then one can choose $w$ of order $4$. The computation above shows that 
\[
\bn=\bn_{A\bN K}(ww^{-2}\til \bn w^2 w^ {-1} ) \nu(w^{-2}\til \bn w^2).
\]
On the other hand, $w=w^{-3}$ so that 
\[
\nu\big(\nu(\bn)\big)=\nu(\til \bn)= \bn_{A\bN K}(w \til \bn w^{-1})= \til n(w^{-1}(w^{-2}\til n w^2) w)= \nu^ {-1}(w^{-2}\til \bn w^2).
\]
Thus we deduce $\nu^4 =\tu{Id}$ and the Jacobian determinant is again equal to 
$1$.
\end{description}
\end{proof}
\providecommand{\bysame}{\leavevmode\hbox to3em{\hrulefill}\thinspace}
\providecommand{\MR}{\relax\ifhmode\unskip\space\fi MR }
% \MRhref is called by the amsart/book/proc definition of \MR.
\providecommand{\MRhref}[2]{%
  \href{http://www.ams.org/mathscinet-getitem?mr=#1}{#2}
}
\providecommand{\href}[2]{#2}

\end{document}